\documentclass[oneside,english,reqno]{amsart}
\usepackage[T1]{fontenc}
\usepackage[latin9]{inputenc}
\usepackage{babel}
\usepackage{refstyle}
\usepackage{float}
\usepackage{mathrsfs}
\usepackage{enumitem}
\usepackage{amstext}
\usepackage{amsthm}
\usepackage{amssymb}
\usepackage{graphicx}
\usepackage{esint}
\usepackage[all]{xy}
\PassOptionsToPackage{normalem}{ulem}
\usepackage{ulem}
\usepackage[unicode=true,pdfusetitle,
 bookmarks=true,bookmarksnumbered=false,bookmarksopen=false,
 breaklinks=false,pdfborder={0 0 1},backref=false,colorlinks=false]
 {hyperref}
\hypersetup{
 colorlinks=true,citecolor=blue,linkcolor=blue,linktocpage=true}

\makeatletter

\AtBeginDocument{\providecommand\subref[1]{\ref{sub:#1}}}
\AtBeginDocument{\providecommand\exaref[1]{\ref{exa:#1}}}
\AtBeginDocument{\providecommand\defref[1]{\ref{def:#1}}}
\AtBeginDocument{\providecommand\figref[1]{\ref{fig:#1}}}
\AtBeginDocument{\providecommand\thmref[1]{\ref{thm:#1}}}
\AtBeginDocument{\providecommand\corref[1]{\ref{cor:#1}}}
\AtBeginDocument{\providecommand\secref[1]{\ref{sec:#1}}}
\AtBeginDocument{\providecommand\lemref[1]{\ref{lem:#1}}}
\RS@ifundefined{subref}
  {\def\RSsubtxt{section~}\newref{sub}{name = \RSsubtxt}}
  {}
\RS@ifundefined{thmref}
  {\def\RSthmtxt{theorem~}\newref{thm}{name = \RSthmtxt}}
  {}
\RS@ifundefined{lemref}
  {\def\RSlemtxt{lemma~}\newref{lem}{name = \RSlemtxt}}
  {}

\numberwithin{equation}{section}
\numberwithin{figure}{section}
\numberwithin{table}{section}
\usepackage{enumitem}		
\theoremstyle{plain}
\newtheorem{thm}{\protect\theoremname}[section]
  \theoremstyle{definition}
  \newtheorem{defn}[thm]{\protect\definitionname}
  \theoremstyle{plain}
  \newtheorem{lem}[thm]{\protect\lemmaname}
  \theoremstyle{remark}
  \newtheorem{rem}[thm]{\protect\remarkname}
  \theoremstyle{definition}
  \newtheorem{example}[thm]{\protect\examplename}
  \theoremstyle{remark}
  \newtheorem*{notation*}{\protect\notationname}
  \theoremstyle{plain}
  \newtheorem{prop}[thm]{\protect\propositionname}
  \theoremstyle{plain}
  \newtheorem{cor}[thm]{\protect\corollaryname}
 \newlist{casenv}{enumerate}{4}
 \setlist[casenv]{leftmargin=*,align=left,widest={iiii}}
 \setlist[casenv,1]{label={{\itshape\ \casename} \arabic*.},ref=\arabic*}
 \setlist[casenv,2]{label={{\itshape\ \casename} \roman*.},ref=\roman*}
 \setlist[casenv,3]{label={{\itshape\ \casename\ \alph*.}},ref=\alph*}
 \setlist[casenv,4]{label={{\itshape\ \casename} \arabic*.},ref=\arabic*}
  \theoremstyle{remark}
  \newtheorem*{acknowledgement*}{\protect\acknowledgementname}

\usepackage{needspace}
\usepackage{bbm}
\allowdisplaybreaks

\newref{lem}{refcmd={Lemma \ref{#1}}}
\newref{thm}{refcmd={Theorem \ref{#1}}}
\newref{cor}{refcmd={Corollary \ref{#1}}}
\newref{sec}{refcmd={Section \ref{#1}}}
\newref{sub}{refcmd={Section \ref{#1}}}
\newref{chap}{refcmd={Chapter \ref{#1}}}
\newref{prop}{refcmd={Proposition \ref{#1}}}
\newref{exa}{refcmd={Example \ref{#1}}}
\newref{tab}{refcmd={Table \ref{#1}}}
\newref{rem}{refcmd={Remark \ref{#1}}}
\newref{def}{refcmd={Definition \ref{#1}}}
\newref{fig}{refcmd={Figure \ref{#1}}}
\newref{claim}{refcmd={Claim \ref{#1}}}

\AtBeginDocument{
  
}

\makeatother

  \providecommand{\acknowledgementname}{Acknowledgement}
  \providecommand{\corollaryname}{Corollary}
  \providecommand{\definitionname}{Definition}
  \providecommand{\examplename}{Example}
  \providecommand{\lemmaname}{Lemma}
  \providecommand{\notationname}{Notation}
  \providecommand{\propositionname}{Proposition}
  \providecommand{\remarkname}{Remark}
 \providecommand{\casename}{Case}
\providecommand{\theoremname}{Theorem}

\begin{document}

\title{Representations of the canonical commutation relations--algebra and
the operators of stochastic calculus}

\author{Palle Jorgensen and Feng Tian}

\address{(Palle E.T. Jorgensen) Department of Mathematic s, The University
of Iowa, Iowa City, IA 52242-1419, U.S.A. }

\email{palle-jorgensen@uiowa.edu}

\urladdr{http://www.math.uiowa.edu/\textasciitilde{}jorgen/}

\address{(Feng Tian) Department of Mathematics, Hampton University, Hampton,
VA 23668, U.S.A.}

\email{feng.tian@hamptonu.edu}

\subjclass[2000]{Primary 81S20, 81S40, 60H07, 47L60, 46N30, 65R10, 58J65, 81S25.}

\keywords{Canonical commutation relations, representations, unbounded operators,
closable operator, unbounded derivations, spectral theory, duality,
Gaussian fields, probability space, stochastic processes, discrete
time, path-space measure, stochastic calculus.}

\maketitle
\pagestyle{myheadings}
\markright{}
\begin{abstract}
We study a family of representations of the canonical commutation
relations (CCR)-algebra (an infinite number of degrees of freedom),
which we call admissible. The family of admissible representations
includes the Fock-vacuum representation. We show that, to every admissible
representation, there is an associated Gaussian stochastic calculus,
and we point out that the case of the Fock-vacuum CCR-representation
in a natural way yields the operators of Malliavin calculus. And we
thus get the operators of Malliavin's calculus of variation from a
more algebraic approach than is common. And we obtain explicit and
natural formulas, and rules, for the operators of stochastic calculus.
Our approach makes use of a notion of symmetric (closable) pairs of
operators. The Fock-vacuum representation yields a maximal symmetric
pair. This duality viewpoint has the further advantage that issues
with unbounded operators and dense domains can be resolved much easier
than what is possible with alternative tools. With the use of CCR
representation theory, we also obtain, as a byproduct, a number of
new results in multi-variable operator theory which we feel are of
independent interest.
\end{abstract}

\tableofcontents{}

\section{Introduction}

Both the study of quantum fields, and of quantum statistical mechanics,
entails families of representations of the canonical commutation relations
(CCRs). In the case of an infinite number of degrees of freedom, it
is known that we have existence of many inequivalent representations
of the CCRs. Among the representations, some describe such things
as a nonrelativistic infinite free Bose gas of uniform density. But
the representations of the CCRs play an equally important role in
the kind of infinite-dimensional analysis currently used in a calculus
of variation approach to Gaussian fields, It\=o integrals, including
the Malliavin calculus. In the literature, the infinite-dimensional
stochastic operators of derivatives and stochastic integrals are usually
taken as the starting point, and the representations of the CCRs are
an afterthought. Here we turn the tables. As a consequence of this,
we are able to obtain a number of explicit results in an associated
multi-variable spectral theory. Some of the issues involved are subtle
because the operators in the representations under consideration are
unbounded (by necessity), and, as a result, one must deal with delicate
issues of domains of families of operators and their extensions.

The representations we study result from the Gelfand-Naimark-Segal
construction (GNS) applied to certain states on the CCR-algebra. Our
conclusions and main results regarding this family of CCR representations
(details below, especially sects \ref{sec:CCR} and \ref{sec:con})
hold in the general setting of Gaussian fields. But for the benefit
of readers, we have also included an illustration dealing with the
simplest case, that of the standard Brownian/Wiener process. Many
arguments in the special case carry over to general Gaussian fields
\emph{mutatis mutandis}. In the Brownian case, our initial Hilbert
space will be $\mathscr{L}=L^{2}\left(0,\infty\right)$.

From the initial Hilbert space $\mathscr{L}$, we build the $*$-algebra
$\mbox{CCR}\left(\mathscr{L}\right)$ as in \subref{ccr}. We will
show that the Fock state on $\mbox{CCR}\left(\mathscr{L}\right)$
corresponds to the Wiener measure $\mathbb{P}$. Moreover the corresponding
representation $\pi$ of $\mbox{CCR}\left(\mathscr{L}\right)$ will
be acting on the Hilbert space $L^{2}\left(\Omega,\mathbb{P}\right)$
in such a way that for every $k$ in $\mathscr{L}$, the operator
$\pi(a(k))$ is the Malliavin derivative in the direction of $k$.
We caution that the representations of the $*$-algebra $\mbox{CCR}\left(\mathscr{L}\right)$
are by unbounded operators, but the operators in the range of the
representations will be defined on a single common dense domain.

Example: There are two ways to think of systems of generators for
the CCR-algebra over a fixed infinite-dimensional Hilbert space (``CCR\textquotedblright{}
is short for canonical commutation relations.): 
\begin{enumerate}
\item[(i)] an infinite-dimensional Lie algebra, or 
\item[(ii)] an associative $*$-algebra. 
\end{enumerate}
With this in mind, (ii) will simply be the universal enveloping algebra
of (i); see \cite{MR0498740}. While there is also an infinite-dimensional
\textquotedblleft Lie\textquotedblright{} group corresponding to (i),
so far, we have not found it as useful as the Lie algebra itself.

All this, and related ideas, supply us with tools for an infinite-dimensional
stochastic calculus. It fits in with what is called Malliavin calculus,
but our present approach is different, and more natural from our point
of view; and as corollaries, we obtain new and explicit results in
multi-variable spectral theory which we feel are of independent interest.

There is one particular representation of the CCR version of (i) and
(ii) which is especially useful for stochastic calculus. In the present
paper, we call this representation the Fock vacuum-state representation.
One way of realizing the representations is abstract: Begin with the
Fock vacuum state (or any other state), and then pass to the corresponding
GNS representation. The other way is to realize the representation
with the use of a choice of a Wiener $L^{2}$-space. We prove that
these two realizations are unitarily equivalent. 

By stochastic calculus we mean stochastic derivatives (e.g., Malliavin
derivatives), and integrals (e.g., It\=o-integrals). The paper begins
with the task of realizing a certain stochastic derivative operator
as a closable operator acting between two Hilbert spaces.

\section{Unbounded operators and the CCR-algebra}

\subsection{Unbounded operators between different Hilbert spaces}

While the theory of unbounded operators has been focused on spectral
theory where it is then natural to consider the setting of linear
\emph{endomorphisms} with dense domain in a fixed Hilbert space; many
applications entail operators between distinct Hilbert spaces, say
$\mathscr{H}_{1}$ and $\mathscr{H}_{2}$. Typically the facts given
about the two differ greatly from one Hilbert space to the next.

Let $\mathscr{H}_{i}$, $i=1,2$, be two complex Hilbert spaces. The
respective inner products will be written $\left\langle \cdot,\cdot\right\rangle _{i}$,
with the subscript to identify the Hilbert space in question. 
\begin{defn}
A linear operator $T$ from $\mathscr{H}_{1}$ to $\mathscr{H}_{2}$
is a pair $\mathscr{D}\subset\mathscr{H}_{1}$, $T$, where $\mathscr{D}$
is a linear subspace in $\mathscr{H}_{1}$, and $T\varphi\in\mathscr{H}_{2}$
is well-defined for all $\varphi\in\mathscr{D}$. 
\end{defn}
We say that $\mathscr{D}=dom\left(T\right)$ is the domain of $T$,
and 
\begin{equation}
\mathscr{G}\left(T\right)=\left\{ \begin{pmatrix}\varphi\\
T\varphi
\end{pmatrix}\:;\:\varphi\in\mathscr{D}\right\} \subset\begin{pmatrix}\underset{\oplus}{\mathscr{H}_{1}}\\
\mathscr{H}_{2}
\end{pmatrix}\label{eq:L2}
\end{equation}
is the graph. 

If the closure $\overline{\mathscr{G}\left(T\right)}$ is the graph
of a linear operator, we say that $T$ is \emph{closable. }By closure,
we shall refer to closure in the norm of $\mathscr{H}_{1}\oplus\mathscr{H}_{2}$,
i.e., 
\begin{equation}
\left\Vert \begin{pmatrix}h_{1}\\
h_{2}
\end{pmatrix}\right\Vert ^{2}=\left\Vert h_{1}\right\Vert _{1}^{2}+\left\Vert h_{2}\right\Vert _{2}^{2},\quad h_{i}\in\mathscr{H}_{i}.\label{eq:L3}
\end{equation}

If $dom\left(T\right)$ is dense in $\mathscr{H}_{1}$, we say that
$T$ is densely defined. 
\begin{defn}
\label{def:adj}Let $\mathscr{H}_{1}\xrightarrow{\;T\;}\mathscr{H}_{2}$
be a densely defined operator, and consider the subspace $dom\left(T^{*}\right)\subset\mathscr{H}_{2}$
defined as follows:
\begin{align}
dom\left(T^{*}\right) & =\Big\{ h_{2}\in\mathscr{H}_{2}\:;\:\exists C=C_{h_{2}}<\infty\:\mbox{s.t. }\nonumber \\
 & \qquad\left|\left\langle T\varphi,h_{2}\right\rangle _{2}\right|\leq C\left\Vert \varphi\right\Vert _{1},\;\forall\varphi\in dom\left(T\right)\Big\}\label{eq:L4}
\end{align}
Then, by Riesz' theorem, there is a unique $h_{1}\in\mathscr{H}_{1}$
s.t. 
\begin{equation}
\left\langle T\varphi,h_{2}\right\rangle _{2}=\left\langle \varphi,h_{1}\right\rangle _{1},\;\mbox{and}\label{eq:L5}
\end{equation}
we set $T^{*}h_{2}=h_{1}$. \end{defn}
\begin{lem}
\label{lem:Tadjoint}Given a densely defined operator $\mathscr{H}_{1}\xrightarrow{\;T\;}\mathscr{H}_{2}$,
then $T$ is closable if and only if $dom\left(T^{*}\right)$ is dense
in $\mathscr{H}_{2}$. \end{lem}
\begin{proof}
See \cite{MR1009163}.\end{proof}
\begin{rem}[Notation and Facts]
~
\begin{enumerate}
\item The abbreviated notation $\xymatrix{\mathscr{H}_{1}\ar@/^{0.5pc}/[r]^{T} & \mathscr{H}_{2}\ar@/^{0.5pc}/[l]^{T^{*}}}
$ will be used when the domains of $T$ and $T^{*}$ are understood
from the context.
\item Let $T$ be an operator $\mathscr{H}_{1}\xrightarrow{\;T\;}\mathscr{H}_{2}$
and $\mathscr{H}_{i}$, $i=1,2$, two given Hilbert spaces. Assume
$\mathscr{D}:=dom\left(T\right)$ is dense in $\mathscr{H}_{1}$,
and that $T$ is \emph{closable}. Then there is a unique \emph{closed}
operator, denoted $\overline{T}$ such that 
\begin{equation}
\mathscr{G}\left(\overline{T}\right)=\overline{\mathscr{G}\left(T\right)}\label{eq:L9}
\end{equation}
where ``---'' on the RHS in (\ref{eq:L9}) refers to norm closure
in $\mathscr{H}_{1}\oplus\mathscr{H}_{2}$, see (\ref{eq:L3}).
\item It may happen that $dom\left(T^{*}\right)=0$. See \exaref{Tadj}
below.
\end{enumerate}
\end{rem}
\begin{example}
\label{exa:Tadj}An operator $T:\mathscr{H}_{1}\longrightarrow\mathscr{H}_{2}$
with dense domain s.t. $dom\left(T^{*}\right)=0$, i.e., ``extremely''
non-closable. 

Set $\mathscr{H}_{i}=L^{2}\left(\mu_{i}\right)$, $i=1,2$, where
$\mu_{1}$ and $\mu_{2}$ are two mutually singular measures on a
fixed locally compact measurable space, say $X$. The space $\mathscr{D}:=C_{c}\left(X\right)$
is dense in both $\mathscr{H}_{1}$ and in $\mathscr{H}_{2}$ with
respect to the two $L^{2}$-norms. Then, the identity mapping $T\varphi=\varphi$,
$\forall\varphi\in\mathscr{D}$, becomes a Hilbert space operator
$\mathscr{H}_{1}\xrightarrow{\;T\;}\mathscr{H}_{2}$. 

Using \defref{adj}, we see that $h_{2}\in L^{2}\left(\mu_{2}\right)$
is in $dom\left(T^{*}\right)$ iff $\exists h_{1}\in L^{2}\left(\mu_{1}\right)$
such that 
\begin{equation}
\int\varphi\,h_{1}\,d\mu_{1}=\int\varphi\,h_{2}\,d\mu_{2},\quad\forall\varphi\in\mathscr{D}.\label{eq:L16}
\end{equation}
Since $\mathscr{D}$ is dense in both $L^{2}$-spaces, we get 
\begin{equation}
\int_{E}h_{1}\,d\mu_{1}=\int_{E}h_{2}\,d\mu_{2},\label{eq:L17}
\end{equation}
where $E=supp\left(\mu_{2}\right)$. 

Now suppose $h_{2}\neq0$ in $L^{2}\left(\mu_{2}\right)$, then there
is a subset $A\subset E$ s.t. $h_{2}>0$ on $A$, $\mu_{2}\left(A\right)>0$,
and $\int_{A}h_{2}\,d\mu_{2}>0$. But $\int_{A}h_{1}\,d\mu_{1}=\int_{A}h_{2}\,d\mu_{2}$,
and $\int_{A}h_{1}\,d\mu_{1}=0$ since $\mu_{1}\left(A\right)=0$.
This contradiction proves that $dom\left(T^{*}\right)=0$; and in
particular $T$ is unbounded and non-closable.\end{example}
\begin{thm}
\label{thm:TPolar}Let $\mathscr{H}_{1}\xrightarrow{\;T\;}\mathscr{H}_{2}$
be a densely defined operator, and assume that $dom\left(T^{*}\right)$
is dense in $\mathscr{H}_{2}$, i.e., $T$ is closable, then both
of the operators $T^{*}\overline{T}$ and $\overline{T}T^{*}$ are
densely defined, and both are selfadjoint. 

Moreover, there is a \uline{partial isometry} $U:\mathscr{H}_{1}\longrightarrow\mathscr{H}_{2}$
with initial space in $\mathscr{H}_{1}$ and final space in $\mathscr{H}_{2}$
such that 
\begin{equation}
T=U\left(T^{*}\overline{T}\right)^{\frac{1}{2}}=\left(\overline{T}T^{*}\right)^{\frac{1}{2}}U.\label{eq:L18}
\end{equation}
(Eq. (\ref{eq:L18}) is called the polar decomposition of $T$.)\end{thm}
\begin{proof}
See, e.g., \cite{MR1009163}.
\end{proof}

\subsection{\label{sub:ccr}The CCR-algebra, and the Fock representations}

There are two $*$-algebras built functorially from a fixed (single)
Hilbert space $\mathscr{L}$; often called the one-particle Hilbert
space (in physics). The dimension $\dim\mathscr{L}$ is called \emph{the
number of degrees of freedom}. The case of interest here is when $\dim\mathscr{L}=\aleph_{0}$
(countably infinite). The two $*$-algebras are called the CAR, and
the CCR-algebras, and they are extensively studied; see e.g., \cite{MR611508}.
Of the two, only CAR($\mathscr{L}$) is a $C^{*}$-algebra. The operators
arising from representations of CCR($\mathscr{L}$) will be \emph{unbounded},
but still having a common dense domain in the respective representation
Hilbert spaces. In both cases, we have a Fock representation. For
CCR($\mathscr{L}$), it is realized in the symmetric Fock space $\Gamma_{sym}\left(\mathscr{L}\right)$.
There are many other representations, inequivalent to the respective
Fock representations. 

Let $\mathscr{L}$ be as above. The CCR($\mathscr{L}$) is generated
axiomatically by a system, $a\left(h\right)$, $a^{*}\left(h\right)$,
$h\in\mathscr{L}$, subject to
\begin{equation}
\begin{split}\left[a\left(h\right),a\left(k\right)\right] & =0,\;\forall h,k\in\mathscr{L},\;\mbox{and}\\
\left[a\left(h\right),a^{*}\left(k\right)\right] & =\left\langle h,k\right\rangle _{\mathscr{L}}\mathbbm{1}.
\end{split}
\label{eq:cr2}
\end{equation}

\begin{notation*}
In (\ref{eq:cr2}), $\left[\cdot,\cdot\right]$ denotes the commutator.
More specifically, if $A,B$ are elements in a $*$-algebra, set $\left[A,B\right]:=AB-BA$.
\end{notation*}
The \emph{Fock States} $\omega_{Fock}$ on the CCR-algebra are specified
as follows:
\begin{align}
\omega_{Fock}\left(a\left(h\right)a^{*}\left(k\right)\right) & =\left\langle h,k\right\rangle _{\mathscr{L}}\label{eq:cr4}
\end{align}
with the vacuum property 
\begin{equation}
\omega_{Fock}\left(a^{*}\left(h\right)a\left(h\right)\right)=0,\;\forall h\in\mathscr{L};\label{eq:cr5}
\end{equation}

For the corresponding Fock representations $\pi$ we have:
\begin{align}
\left[\pi\left(h\right),\pi^{*}\left(k\right)\right] & =\left\langle h,k\right\rangle _{\mathscr{L}}I_{\Gamma_{sym}\left(\mathscr{L}\right)},\label{eq:cr4-1}
\end{align}
where $I_{\Gamma_{sym}\left(\mathscr{L}\right)}$ on the RHS of (\ref{eq:cr4-1})
refers to the identity operator.

Some relevant papers regarding the CCR-algebra and its representations
are \cite{MR0152295,MR0417882,MR0474479,MR0390124,MR0622034,MR0330350,MR887102,MR1110533}.

\subsection{An infinite-dimensional Lie algebra}

Let $\mathscr{L}$ be a separable Hilbert space, i.e., $\dim\mathscr{L}=\aleph_{0}$,
and let $\mbox{CCR}\left(\mathscr{L}\right)$ be the corresponding
CCR-algebra. As above, its generators are denoted $a\left(k\right)$
and $a^{*}\left(l\right)$, for $k,l\in\mathscr{L}$. We shall need
the following:
\begin{prop}
~
\begin{enumerate}
\item \label{enu:cc1}The ``quadratic'' elements in $\mbox{CCR}\left(\mathscr{L}\right)$
of the form $a\left(k\right)a^{*}\left(l\right)$, $k,l\in\mathscr{L}$,
span a Lie algebra $\mathfrak{g}\left(\mathscr{L}\right)$ under the
commutator bracket.
\item We have 
\begin{eqnarray*}
 &  & \left[a\left(h\right)a^{*}\left(k\right),a\left(l\right)a^{*}\left(m\right)\right]\\
 & = & \left\langle h,m\right\rangle _{\mathscr{L}}a\left(l\right)a^{*}\left(k\right)-\left\langle k,l\right\rangle _{\mathscr{L}}a\left(h\right)a^{*}\left(m\right),
\end{eqnarray*}
for all $h,k,l,m\in\mathscr{L}$. 
\item \label{enu:cc3}If $\left\{ \varepsilon_{i}\right\} _{i\in\mathbb{N}}$
is an ONB in $\mathscr{L}$, then the non-zero commutators are as
follows: Set $\gamma_{i,j}:=a\left(\varepsilon_{i}\right)a^{*}\left(\varepsilon_{j}\right)$,
then, for $i\neq j$, we have 
\begin{eqnarray}
\left[\gamma_{i,i},\gamma_{j,i}\right] & = & \gamma_{j,i};\\
\left[\gamma_{i,i},\gamma_{i,j}\right] & = & -\gamma_{i,j};\:\mbox{and}\\
\left[\gamma_{j,i},\gamma_{i,j}\right] & = & \gamma_{i,i}-\gamma_{j,j}.
\end{eqnarray}
All other commutators vanish; in particular, $\left\{ \gamma_{i,i}\mid i\in\mathbb{N}\right\} $
spans an \uline{abelian} sub-Lie algebra in $\mathfrak{g}\left(\mathscr{L}\right)$.

Note further that, when $i\neq j$, then the three elements 
\begin{equation}
\gamma_{i,i}-\gamma_{j,j},\quad\gamma_{i,j},\quad\mbox{and}\quad\gamma_{j,i}\label{eq:li4}
\end{equation}
span (over $\mathbb{R}$) an isomorphic copy of the Lie algebra $sl_{2}\left(\mathbb{R}\right)$.

\item \label{enu:cc4}The Lie algebra generated by the first-order elements
$a\left(h\right)$ and $a^{*}\left(k\right)$ for $h,k\in\mathscr{L}$,
is called the Heisenberg Lie algebra $\mathfrak{h}\left(\mathscr{L}\right)$.
It is normalized by $\mathfrak{g}\left(\mathscr{L}\right)$; indeed
we have:
\begin{align*}
\left[a\left(l\right)a^{*}\left(m\right),a\left(h\right)\right] & =-\left\langle m,h\right\rangle _{\mathscr{L}}a\left(l\right),\;\mbox{and}\\
\left[a\left(l\right)a^{*}\left(m\right),a^{*}\left(k\right)\right] & =\left\langle l,k\right\rangle _{\mathscr{L}}a^{*}\left(m\right),\;\forall l,m,h,k\in\mathscr{L}.
\end{align*}

\end{enumerate}
\end{prop}
\begin{proof}
The verification of each of the four assertions (\ref{enu:cc1})-(\ref{enu:cc4})
uses only the fixed axioms for the CCR, i.e., 
\begin{equation}
\left\{ \begin{split}\left[a\left(k\right),a\left(l\right)\right] & =0,\\
\left[a^{*}\left(k\right),a^{*}\left(l\right)\right] & =0,\;\mbox{and}\\
\left[a\left(k\right),a^{*}\left(l\right)\right] & =\left\langle k,l\right\rangle _{\mathscr{L}}\mathbbm{1},\;k,l\in\mathscr{L};
\end{split}
\right.\label{eq:li5}
\end{equation}
where $\mathbbm{1}$ denotes the unit-element in $\mbox{CCR}\left(\mathscr{L}\right)$.\end{proof}
\begin{cor}
\label{cor:li}Let $\mbox{CCR}\left(\mathscr{L}\right)$ be the CCR-algebra,
generators $a\left(k\right)$, $a^{*}\left(l\right)$, $k,l\in\mathscr{L}$,
and let $\left[\cdot,\cdot\right]$ denote the commutator Lie bracket;
then, for all $k,h_{1},\cdots,h_{n}\in\mathscr{L}$, and all $p\in\mathbb{R}\left[x_{1},\cdots,x_{n}\right]$
(= the $n$-variable polynomials over $\mathbb{R}$), we have
\begin{eqnarray}
 &  & \left[a\left(k\right),p\left(a^{*}\left(h_{1}\right),\cdots,a^{*}\left(h_{n}\right)\right)\right]\nonumber \\
 & = & \sum_{i=1}^{n}\frac{\partial p}{\partial x_{i}}\left(a^{*}\left(h_{1}\right),\cdots,a^{*}\left(h_{n}\right)\right)\left\langle k,h_{i}\right\rangle _{\mathscr{L}}.\label{eq:li6}
\end{eqnarray}
\end{cor}
\begin{proof}
The verification of (\ref{eq:li6}) uses only the axioms for the CCR,
i.e., the commutation relations (\ref{eq:li5}) above, plus a little
combinatorics. 
\end{proof}
We shall now return to a stochastic variation of formula (\ref{eq:li6}),
the so called Malliavin derivative in the direction $k$. In this,
the system $\left(a^{*}\left(h_{1}\right),\cdots,a^{*}\left(h_{n}\right)\right)$
in (\ref{eq:li6}) instead takes the form of a multivariate Gaussian
random variable.

\subsection{\label{sub:SI}Gaussian Hilbert space}

The literature on Gaussian Hilbert space, white noise analysis, and
its relevance to Malliavin calculus is vast; and we limit ourselves
here to citing \cite{MR2052267,MR2793121,MR2966130,MR3155263,MR3231624,MR3402823,MR3424704},
and the papers cited there.

\paragraph*{Setting and Notation.}
\begin{enumerate}[label=]
\item $\mathscr{L}$: a fixed \emph{real} Hilbert space
\item $\left(\Omega,\mathcal{F},\mathbb{P}\right)$: a fixed probability
space
\item $L^{2}\left(\Omega,\mathbb{P}\right)$: the Hilbert space $L^{2}\left(\Omega,\mathcal{F},\mathbb{P}\right)$,
also denoted by $L^{2}\left(\mathbb{P}\right)$
\item $\mathbb{E}$: the mean or expectation functional, where $\mathbb{E}\left(\cdots\right)=\int_{\Omega}\left(\cdots\right)d\mathbb{P}$\end{enumerate}
\begin{defn}
\label{def:GH}Fix a \emph{real} Hilbert space $\mathscr{L}$ and
a given probability space $\left(\Omega,\mathcal{F},\mathbb{P}\right)$.
We say the pair $\left(\mathscr{L},\left(\Omega,\mathcal{F},\mathbb{P}\right)\right)$
is a \emph{Gaussian Hilbert space}.

A \emph{Gaussian field} is a linear mapping $\Phi:\mathscr{L}\longrightarrow L^{2}\left(\Omega,\mathbb{P}\right)$,
such that 
\[
\left\{ \Phi\left(h\right)\mid h\in\mathscr{L}\right\} 
\]
is a Gaussian process indexed by $\mathscr{L}$ satisfying: 
\begin{enumerate}
\item $\mathbb{E}\left(\Phi\left(h\right)\right)=0$, $\forall h\in\mathscr{L}$;
\item $\forall n\in\mathbb{N}$, $\forall l_{1},\cdots,l_{n}\subset\mathscr{L}$,
the random variable $\left(\Phi\left(l_{1}\right),\cdots,\Phi\left(l_{n}\right)\right)$
is jointly Gaussian, with 
\begin{equation}
\mathbb{E}\left(\Phi\left(l_{i}\right)\Phi\left(l_{j}\right)\right)=\left\langle l_{i},l_{j}\right\rangle ,\label{eq:jg}
\end{equation}
i.e., $\left(\left\langle l_{i},l_{j}\right\rangle \right)_{i=1}^{n}$
= the covariance matrix. (For the existence of Gaussian fields, see
the discussion below.)
\end{enumerate}
\end{defn}

\begin{rem}
\label{rem:gauss}For all finite systems $\left\{ l_{i}\right\} \subset\mathscr{L}$,
set $G_{n}=\left(\left\langle l_{i},l_{j}\right\rangle \right)_{i,j=1}^{n}$,
called the Gramian. Assume $G_{n}$ non-singular for convenience,
so that $\det G_{n}\neq0$. Then there is an associated Gaussian density
$g^{\left(G_{n}\right)}$ on $\mathbb{R}^{n}$, 
\begin{equation}
g^{\left(G_{n}\right)}\left(x\right)=\left(2\pi\right)^{-n/2}\left(\det G_{n}\right)^{-1/2}\exp\left(-\frac{1}{2}\left\langle x,G_{n}^{-1}x\right\rangle _{\mathbb{R}^{n}}\right)\label{eq:gd1}
\end{equation}
The condition in (\ref{eq:jg}) assumes that for all continuous functions
$f:\mathbb{R}^{n}\longrightarrow\mathbb{R}$ (e.g., polynomials),
we have 
\begin{equation}
\mathbb{E}\underset{\text{real valued}}{(\underbrace{f\left(\Phi\left(l_{1}\right),\cdots,\Phi\left(l_{n}\right)\right)})}=\int_{\mathbb{R}^{n}}f\left(x\right)g^{\left(G_{n}\right)}\left(x\right)dx;\label{eq:gd2}
\end{equation}
where $x=\left(x_{1},\cdots,x_{n}\right)\in\mathbb{R}^{n}$, and $dx=dx_{1}\cdots dx_{n}$
= Lebesgue measure on $\mathbb{R}^{n}$. See \figref{gauss} for an
illustration.

In particular, for $n=2$, $\left\langle l_{1},l_{2}\right\rangle =\left\langle k,l\right\rangle $,
and $f\left(x_{1},x_{2}\right)=x_{1}x_{2}$, we then get $\mathbb{E}\left(\Phi\left(k\right)\Phi\left(l\right)\right)=\left\langle k,l\right\rangle $,
i.e., the inner product in $\mathscr{L}$.

\begin{figure}
\includegraphics[width=0.6\textwidth]{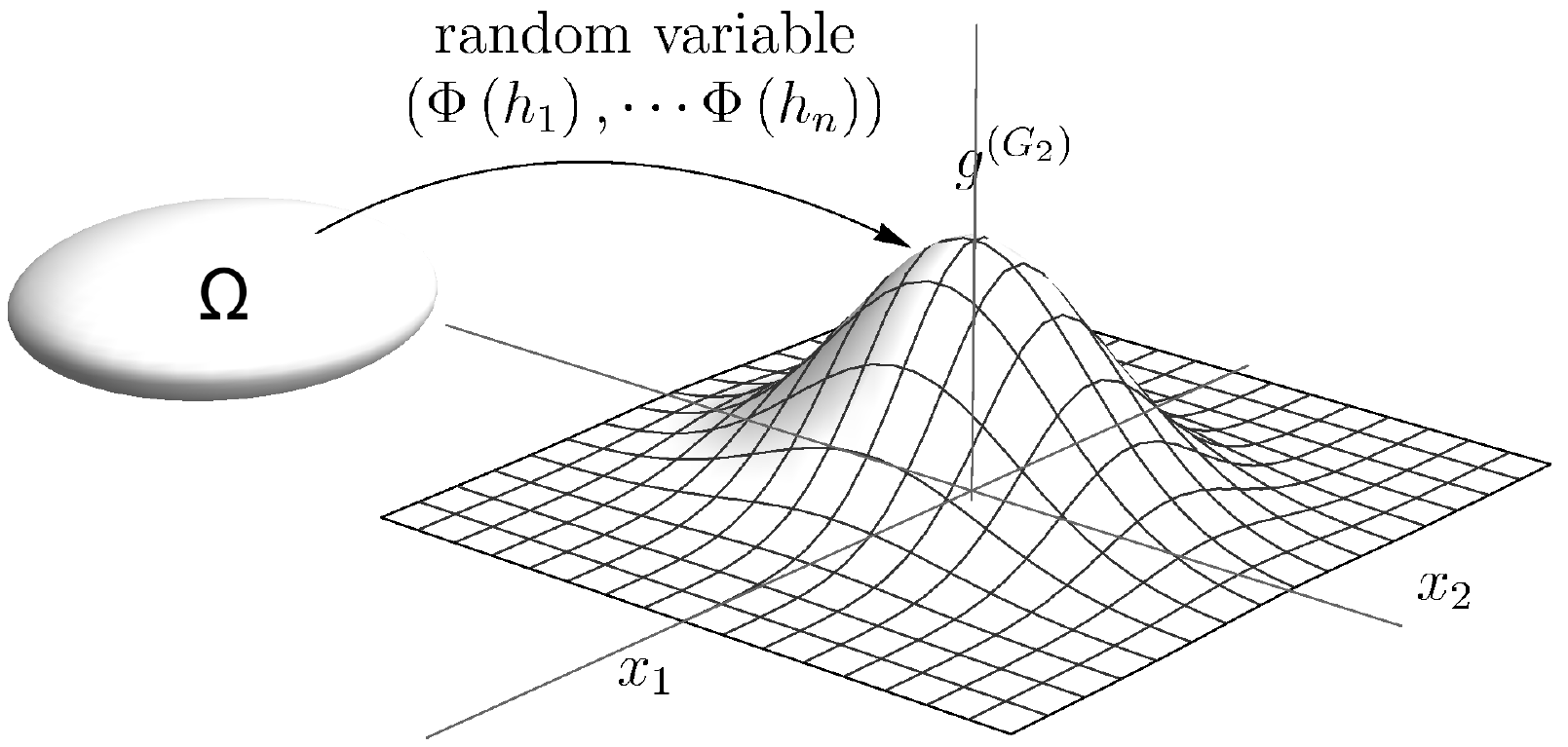}

\caption{\label{fig:gauss}The multivariate Gaussian $\left(\Phi\left(h_{1}\right),\cdots\Phi\left(h_{n}\right)\right)$
and its distribution. The Gaussian with Gramian matrix (Gram matrix)
$G_{n}$, $n=2$.}
\end{figure}

\end{rem}

For our applications, we need the following facts about Gaussian fields.

Fix a Hilbert space $\mathscr{L}$ over $\mathbb{R}$ with inner product
$\left\langle \cdot,\cdot\right\rangle _{\mathscr{L}}$. Then (see
\cite{MR562914,MR3424704,MR0265548}) there is a \emph{probability
space} $\left(\Omega,\mathcal{F},\mathbb{P}\right)$, depending on
$\mathscr{L}$, and a \emph{real} linear mapping $\Phi:\mathscr{L}\longrightarrow L^{2}\left(\Omega,\mathcal{F},\mathbb{P}\right)$,
i.e., a Gaussian field as specified in \defref{GH}, satisfying
\begin{equation}
\mathbb{E}\big(e^{i\Phi\left(k\right)}\big)=e^{-\frac{1}{2}\left\Vert k\right\Vert ^{2}},\quad\forall k\in\mathscr{L}.\label{eq:e1}
\end{equation}

It follows from the literature (see also \cite{1408.1164}) that $\Phi\left(k\right)$
may be thought of as a generalized It\=o-integral. One approach to
this is to select a nuclear Fréchet space $\mathcal{S}$ with dual
$\mathcal{S}'$ such that 
\begin{equation}
\mathcal{S}\hookrightarrow\mathscr{L}\hookrightarrow\mathcal{S}'\label{eq:e0}
\end{equation}
forms a Gelfand triple. In this case we may take $\Omega=\mathcal{S}'$,
and $\Phi\left(k\right)$, $k\in\mathscr{L}$, to be the extension
of the mapping 
\begin{equation}
\mathcal{S}'\ni\omega\longrightarrow\omega\left(\varphi\right)=\left\langle \varphi,\omega\right\rangle \label{eq:e2}
\end{equation}
defined initially only for $\varphi\in\mathcal{S}$, but, with the
use of (\ref{eq:e2}), now extended, via (\ref{eq:e1}), from $\mathcal{S}$
to $\mathscr{L}$. See also \exaref{GH} below. 
\begin{example}
Fix a measure space $\left(X,\mathcal{B},\mu\right)$. Let $\Phi:L^{2}\left(\mu\right)\longrightarrow L^{2}\left(\Omega,\mathbb{P}\right)$
be a Gaussian field such that
\[
\mathbb{E}\left(\Phi_{A}\Phi_{B}\right)=\mu\left(A\cap B\right),\quad\forall A,B\in\mathcal{B}
\]
where $\Phi_{E}:=\Phi\left(\chi_{E}\right)$, $\forall E\in\mathcal{B}$;
and $\chi_{E}$ denotes the characteristic function. In this case,
$\mathscr{L}=L^{2}\left(X,\mu\right)$. 

Then we have $\Phi\left(k\right)=\int_{X}k\left(x\right)d\Phi$, i.e.,
the It\=o-integral, and the following holds:
\begin{equation}
\mathbb{E}\left(\Phi\left(k\right)\Phi\left(l\right)\right)=\left\langle k,l\right\rangle =\int_{X}k\left(x\right)l\left(x\right)d\mu\left(x\right)\label{eq:iso}
\end{equation}
for all $k,l\in\mathscr{L}=L^{2}\left(X,\mu\right)$. Eq. (\ref{eq:iso})
is known as the It\=o-isometry.
\end{example}

\begin{example}[The special case of Brownian motion]
\label{exa:GH} There are many ways of realizing a Gaussian probability
space $\left(\Omega,\mathcal{F},\mathbb{P}\right)$. Two candidates
for the sample space:
\begin{casenv}
\item Standard Brownian motion process: $\Omega=C\left(\mathbb{R}\right)$,
$\mathcal{F}$ = $\sigma$-algebra generated by cylinder sets, $\mathbb{P}$
= Wiener measure. Set $B_{t}\left(\omega\right)=\omega\left(t\right)$,
$\forall\omega\in\Omega$; and $\Phi\left(k\right)=\int_{\mathbb{R}}k\left(t\right)dB_{t}$,
$\forall k\in L^{2}\left(\Omega,\mathbb{P}\right)$.
\item The Gelfand triples: $\mathcal{S}\hookrightarrow L^{2}\left(\mathbb{R}\right)\hookrightarrow\mathcal{S}'$,
where

\begin{enumerate}[label=]
\item $S$ = the Schwartz space of test functions;
\item $\mathcal{S}'$ = the space of tempered distributions. 
\end{enumerate}

Set $\Omega=\mathcal{S}'$, $\mathcal{F}$ = $\sigma$-algebra generated
by cylinder sets of $\mathcal{S}'$, and define 
\[
\Phi\left(k\right):=\widehat{k}\left(\omega\right)=\left\langle k,\omega\right\rangle ,\;k\in L^{2}\left(\mathbb{R}\right),\;\omega\in\mathcal{S}'.
\]
Note $\Phi$ is defined by extending the duality $\mathcal{S}\longleftrightarrow\mathcal{S}'$
to $L^{2}\left(\mathbb{R}\right)$. The probability measure $\mathbb{P}$
is defined from
\[
\mathbb{E}\big(e^{i\left\langle k,\cdot\right\rangle }\big)=\int_{\mathcal{S}'}e^{i\widehat{k}\left(\omega\right)}d\mathbb{P}\left(\omega\right)=e^{-\frac{1}{2}\left\Vert k\right\Vert _{L^{2}\left(\mathbb{R}\right)}^{2}},
\]
by Minlos' theorem \cite{MR562914,MR3424704}. 

\end{casenv}
\end{example}
\begin{defn}
\label{def:poly}Let $\mathscr{D}\subset L^{2}\left(\Omega,\mathcal{F},\mathbb{P}\right)$
be the dense subspace spanned by functions $F$, where $F\in\mathscr{D}$
iff $\exists n\in\mathbb{N}$, $\exists h_{1},\cdots,h_{n}\in\mathscr{L}$,
and $p\in\mathbb{R}\left[x_{1},\cdots,x_{n}\right]$ = the polynomial
ring, such that 
\[
F=p\left(\Phi\left(h_{1}\right),\cdots,\Phi\left(h_{n}\right)\right):\Omega\longrightarrow\mathbb{R}.
\]
(See the diagram below.) The case of $n=0$ corresponds to the constant
function $\mathbbm{1}$ on $\Omega$. Note that $\Phi\left(h_{i}\right)\in L^{2}\left(\Omega,\mathbb{P}\right)$.
\[
\xymatrix{ &  & \mathbb{R}^{n}\ar[drr]^{p}\\
\Omega\ar[rru]^{\left(\Phi\left(h_{1}\right),\cdots,\Phi\left(h_{n}\right)\right)}\ar[rrrr]_{F} &  &  &  & \mathbb{R}
}
\]
\end{defn}
\begin{lem}
\label{lem:dense}The polynomial fields $\mathscr{D}$ in Def. \ref{def:poly}
form a dense subspace in $L^{2}\left(\Omega,\mathbb{P}\right)$. \end{lem}
\begin{proof}
The easiest argument below takes advantage of the isometric isomorphism
of $L^{2}\left(\Omega,\mathbb{P}\right)$ with the symmetric Fock
space 
\[
\Gamma_{sym}\left(\mathscr{L}\right)=\underset{\text{1 dim}}{\underbrace{\mathscr{H}_{0}}}\oplus\sum_{n=1}^{\infty}\underset{\text{\ensuremath{n}-fold symmetric}}{\underbrace{\left(\mathscr{L}\otimes\cdots\otimes\mathscr{L}\right)}}.
\]
For $k_{i}\in\mathscr{L}$, $i=1,2$, there is a unique vector $e^{k_{i}}\in\Gamma_{sym}\left(\mathscr{L}\right)$
such that 
\[
\left\langle e^{k_{1}},e^{k_{2}}\right\rangle _{\Gamma_{sym}\left(\mathscr{L}\right)}=\sum_{n=0}^{\infty}\frac{\left\langle k_{1},k_{2}\right\rangle ^{n}}{n!}=e^{\left\langle k_{1},k_{2}\right\rangle _{\mathscr{L}}}.
\]
Moreover, 
\[
\Gamma_{sym}\left(\mathscr{L}\right)\ni e^{k}\xrightarrow{\;W_{0}\;}e^{\Phi\left(k\right)-\frac{1}{2}\left\Vert k\right\Vert _{\mathscr{L}}^{2}}\in L^{2}\left(\Omega,\mathbb{P}\right)
\]
extends by linearity and closure to a unitary isomorphism $\Gamma_{sym}\left(\mathscr{L}\right)\xrightarrow{\;W\;}L^{2}\left(\Omega,\mathbb{P}\right)$,
mapping onto $L^{2}\left(\Omega,\mathbb{P}\right)$. Hence $\mathscr{D}$
is dense in $L^{2}\left(\Omega,\mathbb{P}\right)$, as $span\left\{ e^{k}\mid k\in\mathscr{L}\right\} $
is dense in $\Gamma_{sym}\left(\mathscr{L}\right)$. \end{proof}
\begin{lem}
Let $\mathscr{L}$ be a real Hilbert space, and let $\left(\Omega,\mathcal{F},\mathbb{P},\Phi\right)$
be an associated Gaussian field. For $n\in\mathbb{N}$, let $\left\{ h_{1},\cdots,h_{n}\right\} $
be a system of \uline{linearly independent} vectors in $\mathscr{L}$.
Then, for polynomials $p\in\mathbb{R}\left[x_{1},\cdots,x_{n}\right]$,
the following two conditions are equivalent:
\begin{align}
p\left(\Phi\left(h_{1}\right),\cdots,\Phi\left(h_{n}\right)\right) & =0\quad\mbox{a.e. on \ensuremath{\Omega} w.r.t \ensuremath{\mathbb{P}}; and}\label{eq:pp1}\\
p\left(x_{1},\cdots,x_{n}\right) & \equiv0,\quad\forall\left(x_{1},\cdots,x_{n}\right)\in\mathbb{R}^{n}.\label{eq:pp2}
\end{align}
\end{lem}
\begin{proof}
Let $G_{n}=\left(\left\langle h_{i},h_{j}\right\rangle \right)_{i,j=1}^{n}$
be the Gramian matrix. We have $\det G_{n}\neq0$. Let $g^{\left(G_{n}\right)}\left(x_{1},\cdots,x_{n}\right)$
be the corresponding Gaussian density; see (\ref{eq:gd1}), and \figref{gauss}.
Then the following are equivalent: 
\begin{enumerate}
\item Eq. (\ref{eq:pp1}) holds;
\item $p\left(\Phi\left(h_{1}\right),\cdots,\Phi\left(h_{n}\right)\right)=0$
in $L^{2}\left(\Omega,\mathcal{F},\mathbb{P}\right)$;
\item $\mathbb{E}\left(\left|p\left(\Phi\left(h_{1}\right),\cdots,\Phi\left(h_{n}\right)\right)\right|^{2}\right)=\int_{\mathbb{R}^{n}}\left|p\left(x\right)\right|^{2}g^{\left(G_{n}\right)}\left(x\right)dx=0$;
\item $p\left(x\right)=0$ a.e. $x$ w.r.t. the Lebesgue measure in $\mathbb{R}^{n}$
;
\item $p\left(x\right)=0$, $\forall x\in\mathbb{R}^{n}$; i.e., (\ref{eq:pp2})
holds.
\end{enumerate}
\end{proof}

\section{\label{sec:MD}The Malliavin derivatives}

Below we give an application of the closability criterion for linear
operators $T$ between different Hilbert spaces $\mathscr{H}_{1}$
and $\mathscr{H}_{2}$, but having dense domain in the first Hilbert
space. In this application, we shall take for $T$ to be the so called
Malliavin derivative. The setting for it is that of the Wiener process.
For the Hilbert space $\mathscr{H}_{1}$ we shall take the $L^{2}$-space,
$L^{2}\left(\Omega,\mathbb{P}\right)$ where $\mathbb{P}$ is generalized
Wiener measure. Below we shall outline the basics of the Malliavin
derivative, and we shall specify the two Hilbert spaces corresponding
to the setting of \thmref{TPolar}. We also stress that the literature
on Malliavin calculus and its applications is vast, see e.g., \cite{MR2052267,MR3424704}.

\textbf{Settings.} It will be convenient for us to work with the \emph{real}
Hilbert spaces. 

Let $\left(\Omega,\mathcal{F},\mathbb{P},\Phi\right)$ be as specified
in \defref{GH}, i.e., we consider the Gaussian field $\Phi$. Fix
a \emph{real} Hilbert space $\mathscr{L}$ with $\dim\mathscr{L}=\aleph_{0}$.
Set $\mathscr{H}_{1}=L^{2}\left(\Omega,\mathbb{P}\right)$, and $\mathscr{H}_{2}=L^{2}\left(\Omega\rightarrow\mathscr{L},\mathbb{P}\right)=L^{2}\left(\Omega,\mathbb{P}\right)\otimes\mathscr{L}$,
i.e., vector valued random variables. 

For $\mathscr{H}_{1}$, the inner product $\left\langle \cdot,\cdot\right\rangle _{\mathscr{H}_{1}}$
is 
\begin{equation}
\left\langle F,G\right\rangle _{\mathscr{H}_{1}}=\int_{\Omega}FG\:d\mathbb{P}=\mathbb{E}\left(FG\right);\label{eq:no1}
\end{equation}
where $\mathbb{E}\left(\cdots\right)=\int_{\Omega}\left(\cdots\right)d\mathbb{P}$
is the mean or expectation functional.

On $\mathscr{H}_{2}$, we have the tensor product inner product: If
$F_{i}\in\mathscr{H}_{1}$, $k_{i}\in\mathscr{L}$, $i=1,2$, then
\begin{align}
\left\langle F_{1}\otimes k_{1},F_{2}\otimes k_{2}\right\rangle _{\mathscr{H}_{2}} & =\left\langle F_{1},F_{2}\right\rangle _{\mathscr{H}_{1}}\left\langle k_{1},k_{2}\right\rangle _{\mathscr{L}}\nonumber \\
 & =\mathbb{E}\left(F_{1}F_{2}\right)\left\langle k_{1},k_{2}\right\rangle _{\mathscr{L}}.\label{eq:no2}
\end{align}
Equivalently, if $\psi_{i}:\Omega\longrightarrow\mathscr{L}$, $i=1,2$,
are measurable functions on $\Omega$, we set 
\begin{equation}
\left\langle \psi_{1},\psi_{2}\right\rangle _{\mathscr{H}_{2}}=\int_{\Omega}\left\langle \psi_{1}\left(\omega\right),\psi_{2}\left(\omega\right)\right\rangle _{\mathscr{L}}d\mathbb{P}\left(\omega\right);\label{eq:no3}
\end{equation}
where it is assumed that 
\begin{equation}
\int_{\Omega}\left\Vert \psi_{i}\left(\omega\right)\right\Vert _{\mathscr{L}}^{2}d\mathbb{P}\left(\omega\right)<\infty,\quad i=1,2.\label{eq:no4}
\end{equation}

\begin{rem}
In the special case of standard Brownian motion, we have $\mathscr{L}=L^{2}\left(0,\infty\right)$,
and set $\Phi\left(h\right)=\int_{0}^{\infty}h\left(t\right)d\Phi_{t}$
(= the It\={o}-integral), for all $h\in\mathscr{L}$. Recall we then
have 
\begin{equation}
\mathbb{E}\left(\left|\Phi\left(h\right)\right|^{2}\right)=\int_{0}^{\infty}\left|h\left(t\right)\right|^{2}dt,\label{eq:mc1}
\end{equation}
or equivalently (the It\={o}-isometry),
\begin{equation}
\left\Vert \Phi\left(h\right)\right\Vert _{L^{2}\left(\Omega,\mathbb{P}\right)}=\left\Vert h\right\Vert _{\mathscr{L}},\quad\forall h\in\mathscr{L}.\label{eq:mc2}
\end{equation}
The consideration above also works in the context of general Gaussian
fields; see \subref{SI}.\end{rem}
\begin{defn}
\label{def:MD}Let $\mathscr{D}$ be the dense subspace in $\mathscr{H}_{1}=L^{2}\left(\Omega,\mathbb{P}\right)$
as in \defref{poly}. The operator $T:\mathscr{H}_{1}\longrightarrow\mathscr{H}_{2}$
(= Malliavin derivative) with $dom\left(T\right)=\mathscr{D}$ is
specified as follows: 

For $F\in\mathscr{D}$, i.e., $\exists n\in\mathbb{N}$, $p\left(x_{1},\cdots,x_{n}\right)$
a polynomial in $n$ real variables, and $h_{1},h_{2},\cdots,h_{n}\in\mathscr{L}$,
where 
\begin{equation}
F=p\left(\Phi\left(h_{1}\right),\cdots,\Phi\left(h_{n}\right)\right)\in L^{2}\left(\Omega,\mathbb{P}\right).\label{eq:mc3}
\end{equation}
Set
\begin{equation}
T\left(F\right)=\sum_{j=1}^{n}\left(\frac{\partial}{\partial x_{j}}p\right)\left(\Phi\left(h_{1}\right),\cdots,\Phi\left(h_{n}\right)\right)\otimes h_{j}\in\mathscr{H}_{2}.\label{eq:mc4}
\end{equation}

\end{defn}

In the following two remarks we outline the argument for why the expression
for $T(F)$ in (\ref{eq:mc4}) is independent of the chosen representation
(\ref{eq:mc3}) for the particular $F$. Recall that $F$ is in the
domain $\mathscr{D}$ of $T$. Without some careful justification,
it is not even clear that $T$, as given, defines a linear operator
on its dense domain $\mathscr{D}$. The key steps in the argument
to follow will be the result (\ref{eq:mc5}) in \thmref{MD} below,
and the discussion to follow. 

There is an alternative argument, based instead on \corref{li}; see
also \secref{con} below.
\begin{rem}
\label{rem:MD}It is non-trivial that the formula in (\ref{eq:mc4})
defines a linear operator. Reason: On the LHS in (\ref{eq:mc4}),
the representation of $F$ from (\ref{eq:mc3}) is not unique. So
we must show that $p\left(\Phi\left(h_{1}\right),\cdots,\Phi\left(h_{n}\right)\right)=0$
$\Longrightarrow$ $\text{RHS}_{\left(\ref{eq:mc4}\right)}=0$ as
well. (The dual pair analysis below (see Def. \ref{def:pair}) is
good for this purpose.)

Suppose $F\in\mathscr{D}$ has two representations corresponding to
systems of vectors $h_{1},\cdots,h_{n}\in\mathscr{L}$, and $k_{1},\cdots,k_{m}\in\mathscr{L}$,
with polynomials $p\in\mathbb{R}\left[x_{1},\cdots,x_{n}\right]$,
and $q\in\mathbb{R}\left[x_{1},\cdots,x_{m}\right]$, where 
\begin{equation}
F=p\left(\Phi\left(h_{1}\right),\cdots,\Phi\left(h_{n}\right)\right)=q\left(\Phi\left(k_{1}\right),\cdots,\Phi\left(k_{m}\right)\right).\label{eq:F1}
\end{equation}
We must then verify the identity:
\begin{equation}
\sum_{i=1}^{n}\frac{\partial p}{\partial x_{i}}\left(\Phi\left(h_{1}\right),\cdots,\Phi\left(h_{n}\right)\right)\otimes h_{i}=\sum_{i=1}^{m}\frac{\partial q}{\partial x_{i}}\left(\Phi\left(k_{1}\right),\cdots,\Phi\left(k_{m}\right)\right)\otimes k_{i}.\label{eq:F2}
\end{equation}

The significance of the next result is the implication (\ref{eq:F1})
$\Longrightarrow$ (\ref{eq:F2}), valid for all choices of representations
of the same $F\in\mathscr{D}$. The conclusion from (\ref{eq:mc5})
in \thmref{MD} is that the following holds for all $l\in\mathscr{L}$:
\[
\mathbb{E}\left(\left\langle \mbox{LHS}_{\left(\ref{eq:F2}\right)},l\right\rangle \right)=\mathbb{E}\left(\left\langle \mbox{RHS}_{\left(\ref{eq:F2}\right)},l\right\rangle \right)=\mathbb{E}\left(F\Phi\left(l\right)\right).
\]
Moreover, with a refinement of the argument, we arrive at the identity
\[
\left\langle \mbox{LHS}_{\left(\ref{eq:F2}\right)}-\mbox{RHS}_{\left(\ref{eq:F2}\right)},G\otimes l\right\rangle _{\mathscr{H}_{2}}=0,
\]
valid for all $G\in\mathscr{D}$, and all $l\in\mathscr{L}$. 

But $span\left\{ G\otimes l\mid G\in\mathscr{D},\:l\in\mathscr{L}\right\} $
is dense in $\mathscr{H}_{2}\left(=L^{2}\left(\mathbb{P}\right)\otimes\mathscr{L}\right)$
w.r.t. the tensor-Hilbert norm in $\mathscr{H}_{2}$ (see (\ref{eq:no2}));
and we get the desired identity (\ref{eq:F2}) for any two representations
of $F$. 
\end{rem}

\begin{rem}
An easy case where (\ref{eq:F1}) $\Longrightarrow$ (\ref{eq:F2})
can be verified ``by hand'':

Let $F=\Phi\left(h\right)^{2}$ with $h\in\mathscr{L}\backslash\left\{ 0\right\} $
fixed. We can then pick the two systems $\left\{ h\right\} $ and
$\left\{ h,h\right\} $ with $p\left(x\right)=x^{2}$, and $q\left(x_{1},x_{2}\right)=x_{1}x_{2}$.
A direct calculus argument shows that $\mbox{LHS}_{\left(\ref{eq:F2}\right)}=\mbox{RHS}_{\left(\ref{eq:F2}\right)}=2\Phi\left(h\right)\otimes h\in\mathscr{H}_{2}$.
\end{rem}
We now resume the argument for the general case. 
\begin{defn}[symmetric pair]
\label{def:sp} For $i=1,2$, let $\mathscr{H}_{i}$ be two Hilbert
spaces, and suppose $\mathscr{D}_{i}\subset\mathscr{H}_{i}$ are given
dense subspaces. 

We say that a pair of operators $\left(S,T\right)$ forms a \emph{symmetric
pair} if $dom\left(T\right)=\mathscr{D}_{1}$, and $dom\left(S\right)=\mathscr{D}_{2}$;
and moreover, 
\begin{equation}
\left\langle Tu,v\right\rangle _{\mathscr{H}_{2}}=\left\langle u,Sv\right\rangle _{\mathscr{H}_{1}}\label{eq:sp1}
\end{equation}
holds for $\forall u\in\mathscr{D}_{1}$, $\forall v\in\mathscr{D}_{2}$. 

It is immediate that (\ref{eq:sp1}) may be rewritten in the form
of containment of graphs:
\[
T\subset S^{*},\quad S\subset T^{*}.
\]
In that case, both $S$ and $T$ are \emph{closable}. We say that
a symmetric pair is \emph{maximal} if $\overline{T}=S^{*}$ and $\overline{S}=T^{*}$.
\[
\xymatrix{\mathscr{H}_{1}\ar@/^{1.2pc}/[rr]^{T} &  & \mathscr{H}_{2}\ar@/^{1.2pc}/[ll]^{S}}
\]

\end{defn}
We will establish the following two assertions: 
\begin{enumerate}
\item Indeed $T$ from \defref{MD} is a well-defined linear operator from
$\mathscr{H}_{1}$ to $\mathscr{H}_{2}$ .
\item Moreover, $\left(S,T\right)$ is a maximal symmetric pair (see Definitions
\ref{def:sp}, \ref{def:pair}).\end{enumerate}
\begin{defn}
\label{def:pair}Let $\mathscr{H}_{1}\xrightarrow{\;T\;}\mathscr{H}_{2}$
be the Malliavin derivative with $\mathscr{D}_{1}=dom\left(T\right)$,
see \defref{MD}. Set $\mathscr{D}_{2}=\mathscr{D}_{1}\otimes\mathscr{L}$
= algebraic tensor product, and on $dom\left(S\right)=\mathscr{D}_{2}$,
set 
\[
S\left(F\otimes k\right)=-\left\langle T\left(F\right),k\right\rangle +M_{\Phi\left(k\right)}F,\;\forall F\otimes k\in\mathscr{D}_{2},
\]
where $M_{\Phi\left(k\right)}$ = the operator of multiplication by
$\Phi\left(k\right)$. 
\end{defn}
Note that both operators $S$ and $T$ are linear and well defined
on their respective dense domains, $\mathscr{D}_{i}\subset\mathscr{H}_{i}$,
$i=1,2$. For density, see \lemref{dense}.

It is a \textquotedblleft modern version\textquotedblright{} of ideas
in the literature on analysis of Gaussian processes; but we are adding
to it, giving it a twist in the direction of multi-variable operator
theory, representation theory, and especially to representations of
infinite-dimensional algebras on generators and relations. Moreover
our results apply to more general Gaussian processes than covered
so far.
\begin{lem}
\label{lem:ST}Let $\left(S,T\right)$ be the pair of operators specified
above in \defref{pair}. Then it is a symmetric pair, i.e., 
\[
\left\langle Tu,v\right\rangle _{\mathscr{H}_{2}}=\left\langle u,Sv\right\rangle _{\mathscr{H}_{1}},\quad\forall u\in\mathscr{D}_{1},\;\forall v\in\mathscr{D}_{2}.
\]
Equivalently, 
\[
\left\langle T\left(F\right),G\otimes k\right\rangle _{\mathscr{H}_{2}}=\left\langle F,S\left(G\otimes k\right)\right\rangle _{\mathscr{H}_{1}},\quad\forall F,G\in\mathscr{D},\;\forall k\in\mathscr{L}.
\]

In particular, we have $S\subset T^{*}$, and $T\subset S^{*}$(containment
of graphs.) Moreover, the two operators $S^{*}\overline{S}$ and $T^{*}\overline{T}$
are selfadjoint. (For the last conclusion in the lemma, see \thmref{TPolar}.)\end{lem}
\begin{thm}
\label{thm:MD}Let $T:\mathscr{H}_{1}\longrightarrow\mathscr{H}_{2}$
be the Malliavin derivative, i.e., $T$ is an unbounded closable operator
with dense domain $\mathscr{D}$ consisting of the span of all the
functions $F$ from (\ref{eq:mc3}). Then, for all $F\in dom\left(T\right)$,
and $k\in\mathscr{L}$, we have 
\begin{equation}
\mathbb{E}\left(\left\langle T\left(F\right),k\right\rangle _{\mathscr{L}}\right)=\mathbb{E}\left(F\Phi\left(k\right)\right).\label{eq:mc5}
\end{equation}
\end{thm}
\begin{proof}
We shall prove (\ref{eq:mc5}) in several steps. Once (\ref{eq:mc5})
is established, then there is a recursive argument which yields a
dense subspace in $\mathscr{H}_{2}$, contained in $dom\left(T^{*}\right)$;
and so $T$ is closable. 

Moreover, formula (\ref{eq:mc5}) yields directly the evaluation of
$T^{*}:\mathscr{H}_{2}\longrightarrow\mathscr{H}_{1}$ as follows:
If $k\in\mathscr{L}$, set $\mathbbm{1}\otimes k\in\mathscr{H}_{2}$
where  $\mathbbm{1}$ denotes the constant function ``one'' on $\Omega$.
We get 
\begin{equation}
T^{*}\left(\mathbbm{1}\otimes k\right)=\Phi\left(k\right)=\int_{0}^{\infty}k\left(t\right)d\Phi_{t}\left(=\mbox{the It\={o}-integral}.\right)\label{eq:mc6}
\end{equation}
The same argument works for any Gaussian field; see \defref{GH}.
We refer to the literature \cite{MR2052267,MR3424704} for details.

The proof of (\ref{eq:mc5}) works for any Gaussian process $\mathscr{L}\ni k\longrightarrow\Phi\left(k\right)$
indexed by an arbitrary Hilbert space $\mathscr{L}$ with the inner
product $\left\langle k,l\right\rangle _{\mathscr{L}}$ as the covariance
kernel. 

Formula (\ref{eq:mc5}) will be established as follows: Let $F$ and
$T\left(F\right)$ be as in (\ref{eq:mc3})-(\ref{eq:mc4}).

\textbf{Step 1.} For every $n\in\mathbb{N}$, the polynomial ring
$\mathbb{R}\left[x_{1},x_{2},\cdots,x_{n}\right]$ is invariant under
matrix substitution $y=Mx$, where $M$ is an $n\times n$ matrix
over $\mathbb{R}$. 

\textbf{Step 2.} Hence, in considering (\ref{eq:mc5}) for $\left\{ h_{i}\right\} _{i=1}^{n}\subset\mathscr{L}$,
$h_{1}=k$, we may diagonalize the $n\times n$ Gram matrix $\left(\left\langle h_{i},h_{j}\right\rangle \right)_{i,j=1}^{n}$;
thus without loss of generality, we may assume that the system $\left\{ h_{i}\right\} _{i=1}^{n}$
is orthogonal and normalized, i.e., that 
\begin{equation}
\left\langle h_{i},h_{j}\right\rangle =\delta_{ij},\;\forall i,j\in\left\{ 1,\cdots,n\right\} ,\label{eq:p5}
\end{equation}
and we may take $k=h_{1}$ in $\mathscr{L}$. 

\textbf{Step 3.} With this simplification, we now compute the LHS
in (\ref{eq:mc5}). We note that the joint distribution of $\left\{ \Phi\left(h_{i}\right)\right\} _{i=1}^{n}$
is thus the standard Gaussian kernel in $\mathbb{R}^{n}$, i.e., 
\begin{equation}
g_{n}\left(x\right)=\left(2\pi\right)^{-n/2}e^{-\frac{1}{2}\sum_{i=1}^{n}x_{i}^{2}},\label{eq:p6}
\end{equation}
with $x=\left(x_{1},\cdots,x_{n}\right)\in\mathbb{R}^{n}$. We have
\begin{equation}
x_{1}g_{n}\left(x\right)=-\frac{\partial}{\partial x_{1}}g_{n}\left(x\right)\label{eq:p7}
\end{equation}
by calculus. 

\textbf{Step 4.} A direct computation yields 
\begin{eqnarray*}
\mbox{LHS}_{\left(\ref{eq:mc5}\right)} & = & \mathbb{E}\left(\left\langle T\left(F\right),h_{1}\right\rangle _{\mathscr{L}}\right)\\
 & \underset{\text{by \ensuremath{\left(\ref{eq:p5}\right)}}}{=} & \mathbb{E}\left(\frac{\partial p}{\partial x_{1}}\left(\Phi\left(h_{1}\right),\cdots\Phi\left(h_{n}\right)\right)\right)\\
 & \underset{\text{by \ensuremath{\left(\ref{eq:p6}\right)}}}{=} & \int_{\mathbb{R}^{n}}\frac{\partial p}{\partial x_{1}}\left(x_{1},\cdots,x_{n}\right)g_{n}\left(x_{1},\cdots,x_{n}\right)dx_{1}\cdots dx_{n}\\
 & \underset{\text{int. by parts}}{=} & -\int_{\mathbb{R}^{n}}p\left(x_{1},\cdots,x_{n}\right)\frac{\partial g_{n}}{\partial x_{1}}\left(x_{1},\cdots,x_{n}\right)dx_{1}\cdots dx_{n}\\
 & \underset{\text{by \ensuremath{\left(\ref{eq:p7}\right)}}}{=} & \int_{\mathbb{R}^{n}}x_{1}p\left(x_{1},\cdots,x_{n}\right)g_{n}\left(x_{1},\cdots,x_{n}\right)dx_{1}\cdots dx_{n}\\
 & \underset{\text{by \ensuremath{\left(\ref{eq:p5}\right)}}}{=} & \mathbb{E}\left(\Phi\left(h_{1}\right)p\left(\Phi\left(h_{1}\right),\cdots,\Phi\left(h_{n}\right)\right)\right)\\
 & = & \mathbb{E}\left(\Phi\left(h_{1}\right)F\right)=\mbox{RHS}_{\left(\ref{eq:mc5}\right)},
\end{eqnarray*}
which is the desired conclusion (\ref{eq:mc5}).
\end{proof}

\begin{cor}
Let $\mathscr{H}_{1}$, $\mathscr{H}_{2}$, and $\mathscr{H}_{1}\xrightarrow{\;T\;}\mathscr{H}_{2}$
be as in \thmref{MD}, i.e., $T$ is the Malliavin derivative. Then,
for all $h,k\in\mathscr{L}=L^{2}\left(0,\infty\right)$, we have for
the closure $\overline{T}$ of $T$ the following:
\begin{gather}
\overline{T}(e^{\Phi\left(h\right)})=e^{\Phi\left(h\right)}\otimes h,\quad\mbox{and}\label{eq:Texp1}\\
\mathbb{E}\big(\langle\overline{T}(e^{\Phi\left(h\right)}),k\rangle_{\mathscr{L}}\big)=e^{\frac{1}{2}\left\Vert h\right\Vert _{\mathscr{L}}^{2}}\left\langle h,k\right\rangle _{\mathscr{L}}.\label{eq:Texp2}
\end{gather}
Here $\overline{T}$ denotes the graph-closure of $T$. 

Moreover, 
\begin{equation}
T^{*}\overline{T}(e^{\Phi\left(k\right)})=\left(\Phi\left(k\right)-\left\Vert k\right\Vert _{\mathscr{L}}^{2}\right)e^{\Phi\left(k\right)}.\label{eq:Texp3}
\end{equation}
\end{cor}
\begin{proof}
Eqs. (\ref{eq:Texp1})-(\ref{eq:Texp2}) follow immediately from (\ref{eq:mc5})
and a polynomial approximation to 
\[
e^{x}=\lim_{n\rightarrow\infty}\sum_{0}^{n}\frac{x^{j}}{j!},\quad x\in\mathbb{R};
\]
see (\ref{eq:mc3}). In particular, $e^{\Phi\left(h\right)}\in dom\left(\overline{T}\right)$,
and $\overline{T}\left(e^{\Phi\left(h\right)}\right)$ is well defined.

For (\ref{eq:Texp3}), we use the facts for the Gaussians: 
\[
\mathbb{E}(e^{\Phi\left(k\right)})=e^{\frac{1}{2}\left\Vert k\right\Vert ^{2}},\;\mbox{and}
\]
\[
\mathbb{E}(\Phi\left(k\right)e^{\Phi\left(k\right)})=\left\Vert k\right\Vert ^{2}e^{\frac{1}{2}\left\Vert k\right\Vert ^{2}}.
\]
\end{proof}
\begin{example}
Let $F=\Phi\left(k\right)^{k}$, $\left\Vert k\right\Vert =1$. We
have 
\begin{align*}
T\Phi\left(k\right)^{n} & =n\Phi\left(k\right)^{n-1}\otimes k\\
T^{*}T\Phi\left(k\right)^{n} & =-n\left(n-1\right)\Phi\left(k\right)^{n-2}+n\Phi\left(k\right)^{n}
\end{align*}
and similarly, 
\begin{align*}
\overline{T}e^{\Phi\left(k\right)} & =e^{\Phi\left(k\right)}\otimes k\\
T^{*}\overline{T}e^{\Phi\left(k\right)} & =e^{\Phi\left(k\right)}\left(\Phi\left(k\right)-1\right).
\end{align*}

\end{example}
Let $\left(S,T\right)$ be the symmetric pair, we then have the inclusion
$\overline{T}\subset S^{*}$, i.e., containment of the operator graphs,
$\mathscr{G}\left(\overline{T}\right)\subset\mathscr{G}\left(S^{*}\right)$.
In fact, we have 
\begin{cor}
\label{cor:msym}$\overline{T}=S^{*}$.\end{cor}
\begin{proof}
We will show that $\mathscr{G}\left(S^{*}\right)\ominus\mathscr{G}\left(\overline{T}\right)=0$,
where $\ominus$ stands for the orthogonal complement in the direct
sum-inner product of $\mathscr{H}_{1}\oplus\mathscr{H}_{2}$. Recall
that $\mathscr{H}_{1}=L^{2}\left(\Omega,\mathbb{P}\right)$, and $\mathscr{H}_{2}=\mathscr{H}_{1}\otimes\mathscr{L}$. 

Using (\ref{eq:Texp1}), we will prove that if $F\in dom\left(S^{*}\right)$,
and 
\[
\left\langle \begin{pmatrix}e^{\Phi\left(k\right)}\\
e^{\Phi\left(k\right)}\otimes k
\end{pmatrix},\begin{pmatrix}F\\
S^{*}F
\end{pmatrix}\right\rangle =0,\;\forall k\in\mathscr{L}\Longrightarrow F=0,
\]
which is equivalent to 
\begin{equation}
\mathbb{E}\left(e^{\Phi\left(k\right)}\left(F+\left\langle S^{*}F,k\right\rangle \right)\right)=0,\;\forall k\in\mathscr{L}.\label{eq:t1}
\end{equation}
But it is know that for the Gaussian filed, $span\left\{ e^{\Phi\left(k\right)}\mid k\in\mathscr{L}\right\} $
is dense in $\mathscr{H}_{1}$, and so (\ref{eq:t1}) implies that
$F=0$, which is the desired conclusion.

We can finish the proof of the corollary with an application of Girsanov's
theorem, see e.g., \cite{MR2052267} and \cite{MR2754322}. By this
result, we have a measurable action $\tau$ of $\mathscr{L}$ on $\left(\Omega,\mathcal{F},\mathbb{P}\right)$,
i.e.,
\begin{gather}
\begin{split}\begin{split}\mathscr{L}\end{split}
\xrightarrow{\;\tau\;}Aut\left(\Omega,\mathcal{F}\right)\\
\tau_{k}\circ\tau_{l}=\tau_{k+l}\quad\text{a.e. on }\Omega, & \;\forall k,l\in\mathscr{L}
\end{split}
\label{eq:a1}
\end{gather}
(see also sect \ref{sec:con} below) s.t. $\tau_{k}\left(\mathcal{F}\right)=\mathcal{F}$
for all $k\in\mathscr{L}$, and 
\[
\mathbb{P}\circ\tau_{k}^{-1}\ll\mathbb{P}
\]
with 
\begin{equation}
\frac{d\mathbb{P}\circ\tau_{k}^{-1}}{d\mathbb{P}}=e^{-\frac{1}{2}\left\Vert k\right\Vert _{\mathscr{L}}^{2}}e^{\Phi\left(k\right)},\quad\text{a.e. on }\Omega.\label{eq:a2}
\end{equation}

Returning to (\ref{eq:t1}). An application of (\ref{eq:a2}) to (\ref{eq:t1})
yields:
\begin{equation}
F\left(\cdot+k\right)+\left\langle S^{*}\left(F\right)\left(\cdot+k\right),k\right\rangle _{\mathscr{L}}=0\quad\text{a.e. on }\Omega;\label{eq:a4}
\end{equation}
where we have used ``$\cdot+k$'' for the action in (\ref{eq:a1}).
Since $\tau$ in (\ref{eq:a1}) is an action by measure-automorphisms,
(\ref{eq:a4}) implies 
\begin{equation}
F\left(\cdot\right)+\left\langle S^{*}\left(F\right)\left(\cdot\right),k\right\rangle _{\mathscr{L}}=0;\label{eq:a5}
\end{equation}
again with $k\in\mathscr{L}$ arbitrary. If $F\neq0$ in $L^{2}\left(\Omega,\mathcal{F},\mathbb{P}\right)$,
then the second term in (\ref{eq:a5}) would be independent of $k$
which is impossible with $S^{*}\left(F\right)\left(\cdot\right)\neq0$.
But if $S^{*}\left(F\right)=0$, then $F\left(\cdot\right)=0$ (in
$L^{2}\left(\Omega,\mathcal{F},\mathbb{P}\right)$) by (\ref{eq:a5});
and so the proof is completed.\end{proof}
\begin{rem}
We recall the definition of the domain of the closure $\overline{T}$.
The following is a necessary and sufficient condition for an $F\in L^{2}\left(\Omega,\mathcal{F},\mathbb{P}\right)$
to be in the domain of $\overline{T}$:

$F\in dom\left(\overline{T}\right)$ $\Longleftrightarrow$ $\exists$
a sequence $\left\{ F_{n}\right\} \subset\mathscr{D}$ s.t. 
\begin{equation}
\lim_{n,m\rightarrow\infty}\mathbb{E}\left(\left|F_{n}-F_{m}\right|^{2}+\left\Vert T\left(F_{n}\right)-T\left(F_{m}\right)\right\Vert _{\mathscr{L}}^{2}\right)=0.\label{eq:aa1}
\end{equation}
When (\ref{eq:aa1}) holds, we have:
\begin{equation}
\overline{T}\left(F\right)=\lim_{n\rightarrow\infty}T\left(F_{n}\right)\label{eq:aa2}
\end{equation}
where the limit on the RHS in (\ref{eq:aa2}) is in the Hilbert norm
of $L^{2}\left(\Omega,\mathcal{F},\mathbb{P}\right)\otimes\mathscr{L}$.\end{rem}
\begin{cor}
Let $\left(\mathscr{L},\Omega,\mathcal{F},\mathbb{P},\Phi\right)$
be as above, and let $T$ and $S$ be the two operators from \corref{msym}.
Then, for the domain of $\overline{T}$, we have the following:

For random variables $F$ in $L^{2}\left(\Omega,\mathcal{F},\mathbb{P}\right)$,
the following two conditions are equivalent:
\begin{enumerate}
\item $F\in dom\left(\overline{T}\right)$;
\item $\exists C=C_{F}<\infty$ s.t. 
\[
\left|\mathbb{E}\left(F\,S\left(\psi\right)\right)\right|^{2}\leq C\,\mathbb{E}\left(\left\Vert \psi\left(\cdot\right)\right\Vert _{\mathscr{L}}^{2}\right)
\]
holds for $\forall\psi\in span\left\{ G\otimes k\mid G\in\mathscr{D},\:k\in\mathscr{L}\right\} $.

Recall
\[
S\left(\cdot\otimes k\right)=M_{\Phi\left(k\right)}\cdot-\left\langle T\left(\cdot\right),k\right\rangle _{\mathscr{L}};
\]
equivalently, 
\[
S\left(G\otimes k\right)=\Phi\left(k\right)G-\left\langle T\left(G\right),k\right\rangle _{\mathscr{L}}
\]
for all $G\in\mathscr{D}$, and all $k\in\mathscr{L}$.

\end{enumerate}
\end{cor}
\begin{proof}
Immediate from the previous corollary.
\end{proof}

\subsection{A derivation on the algebra $\mathscr{D}$}

The study of unbounded derivations has many applications in mathematical
physics; in particular in making precise the time dependence of quantum
observables, i.e., the dynamics in the Schrödinger picture; --- in
more detail, in the problem of constructing dynamics in statistical
mechanics. An early application of unbounded derivations (in the commutative
case) can be found in the work of Silov \cite{MR0027131}; and the
later study of unbounded derivations in non-commutative $C^{*}$-algebras
is outlined in \cite{MR611508}. There is a rich in variety unbounded
derivations, because of the role they play in applications to dynamical
systems in quantum physics. 

But previously the theory of unbounded derivations has not yet been
applied systematically to stochastic analysis in the sense of Malliavin.
In the present section, we turn to this. We begin with the following:
\begin{lem}[Leibniz-Malliavin]
\label{lem:md}Let $\mathscr{H}_{1}\xrightarrow{\;T\;}\mathscr{H}_{2}$
be the Malliavin derivative from (\ref{eq:mc3})-(\ref{eq:mc4}).
Then, 
\begin{enumerate}
\item $dom\left(T\right)=:\mathscr{D}$, given by (\ref{eq:mc3}), is an
\emph{algebra} of functions on $\Omega$ under pointwise product,
i.e., $FG\in\mathscr{D}$, $\forall F,G\in\mathscr{D}$.
\item $\mathscr{H}_{2}$ is a \emph{module} over $\mathscr{D}$ where $\mathscr{H}_{2}=L^{2}\left(\Omega,\mathbb{P}\right)\otimes\mathscr{L}$
(= vector valued $L^{2}$-random variables.)
\item Moreover, 
\begin{equation}
T\left(FG\right)=T\left(F\right)G+F\,T\left(G\right),\quad\forall F,G\in\mathscr{D},\label{eq:mc7}
\end{equation}
i.e., $T$ is a module-derivation. \end{enumerate}
\begin{notation*}
The eq. (\ref{eq:mc7}) is called the Leibniz-rule. By the Leibniz,
we refer to the traditional rule of Leibniz for the derivative of
a product. And the Malliavin derivative is thus an infinite-dimensional
extension of Leibniz calculus.
\end{notation*}
\end{lem}
\begin{proof}
To show that $\mathscr{D}\subset\mathscr{H}_{1}=L^{2}\left(\Omega,\mathbb{P}\right)$
is an algebra under pointwise multiplication, the following trick
is useful. It follows from finite-dimensional Hilbert space geometry. 

Let $F,G$ be as in \defref{poly}. Then $\exists p,q\in\mathbb{R}\left[x_{1},\cdots,x_{n}\right]$,
$\left\{ l_{i}\right\} _{i=1}^{n}\subset\mathscr{L}$, such that 
\[
F=p\left(\Phi\left(l_{1}\right),\cdots,\Phi\left(l_{n}\right)\right),\;\mbox{and}\quad G=q\left(\Phi\left(l_{1}\right),\cdots,\Phi\left(l_{n}\right)\right).
\]
That is, the same system $l_{1},\cdots,l_{n}$ may be chosen for the
two functions $F$ and $G$.

For the pointwise product, we have 
\[
FG=\left(pq\right)\left(\Phi\left(l_{1}\right),\cdots,\Phi\left(l_{n}\right)\right),
\]
i.e., the product in $\mathbb{R}\left[x_{1},\cdots,x_{n}\right]$
with substitution of the random variable 
\[
\left(\Phi\left(l_{1}\right),\cdots,\Phi\left(l_{n}\right)\right):\Omega\longrightarrow\mathbb{R}^{n}.
\]

Eq. (\ref{eq:mc7}) $\Longleftrightarrow$ $\frac{\partial\left(pq\right)}{\partial x_{i}}=\frac{\partial p}{\partial x_{i}}q+p\frac{\partial q}{\partial x_{i}}$,
which is the usual Leibniz rule applied to polynomials. Note that
\[
T\left(FG\right)=\sum_{i=1}^{n}\frac{\partial}{\partial x_{i}}\left(pq\right)\left(\Phi\left(l_{1}\right),\cdots,\Phi\left(l_{n}\right)\right)\otimes l_{i}.
\]
\end{proof}
\begin{rem}
There is an extensive literature on the theory of densely defined
unbounded derivations in $C^{*}$-algebras. This includes both the
cases of abelian and non-abelian $*$-algebras. And moreover, this
study includes both derivations in these algebras, as well as the
parallel study of module derivations. So the case of the Malliavin
derivative is in fact a special case of this study. Readers interested
in details are referred to \cite{MR1490835}, \cite{MR762697}, \cite{MR545651},
and \cite{MR611508}.\end{rem}
\begin{defn}
Let $\left(\mathscr{L},\Omega,\mathcal{F},\mathbb{P},\Phi\right)$
be a Gaussian field, and $T$ be the Malliavin derivative with $dom\left(T\right)=\mathscr{D}$.
For all $k\in\mathscr{L}$, set 
\begin{equation}
T_{k}\left(F\right):=\left\langle T\left(F\right),k\right\rangle ,\quad F\in\mathscr{D}.\label{eq:MD1}
\end{equation}
In particular, let $F=p\left(\Phi\left(l_{1}\right),\cdots,\Phi\left(l_{1}\right)\right)$
be as in (\ref{eq:mc3}), then 
\[
T_{k}\left(F\right)=\sum_{i=1}^{n}\frac{\partial p}{\partial x_{i}}\left(\Phi\left(l_{1}\right),\cdots,\Phi\left(l_{1}\right)\right)\left\langle l_{i},k\right\rangle .
\]
\end{defn}
\begin{cor}
$T_{k}$ is a derivative on $\mathscr{D}$, i.e., 
\begin{equation}
T_{k}\left(FG\right)=\left(T_{k}F\right)G+F\left(T_{k}G\right),\quad\forall F,G\in\mathscr{D},\;\forall k\in\mathscr{L}.\label{eq:dd1}
\end{equation}
\end{cor}
\begin{proof}
Follows from (\ref{eq:mc7}).
\end{proof}

\begin{cor}
Let $\left(\mathscr{L},\Omega,\mathcal{F},\mathbb{P},\Phi\right)$
be a Gaussian field. Fix $k\in\mathscr{L}$, and let $T_{k}$ be the
Malliavin derivative in the $k$ direction. Then on $\mathscr{D}$
we have 
\begin{equation}
T_{k}+T_{k}^{*}=M_{\Phi\left(k\right)},\;\mbox{and}\label{eq:TTadj1}
\end{equation}
\begin{equation}
\left[T_{k},T_{l}^{*}\right]=\left\langle k,l\right\rangle _{\mathscr{L}}I_{L^{2}\left(\Omega,\mathbb{P}\right)}.\label{eq:TTadj2}
\end{equation}
\end{cor}
\begin{proof}
For all $F,G\in\mathscr{D}$, we have 
\begin{eqnarray*}
\mathbb{E}\left(T_{k}\left(F\right)G\right)+\mathbb{E}\left(F\,T_{k}\left(G\right)\right) & \underset{\text{by \ensuremath{\left(\ref{eq:dd1}\right)}}}{=} & \mathbb{E}\left(T_{k}\left(FG\right)\right)\\
 & \underset{\text{by \ensuremath{\left(\ref{eq:mc5}\right)}}}{=} & \mathbb{E}\left(\Phi\left(k\right)FG\right)
\end{eqnarray*}
which yields the assertion in (\ref{eq:TTadj1}). Eq. (\ref{eq:TTadj2})
now follows from (\ref{eq:TTadj1}) and the fact that $\left[T_{k},T_{l}\right]=0$. \end{proof}
\begin{defn}
Let $\left(\mathscr{L},\Omega,\mathcal{F},\mathbb{P},\Phi\right)$
be a Gaussian field. For all $k\in\mathscr{L}$, let $T_{k}$ be Malliavin
derivative in the $k$-direction (eq. (\ref{eq:MD1})). Assume $\mathscr{L}$
is separable, i.e., $\dim\mathscr{L}=\aleph_{0}$. For every ONB $\left\{ e_{i}\right\} _{i=1}^{\infty}$
in $\mathscr{L}$, let 
\begin{equation}
N:=\sum_{i}T_{e_{i}}^{*}T_{e_{i}}.\label{eq:MN}
\end{equation}
($N$ is the CCR number operator. See \secref{CCR} below.)\end{defn}
\begin{example}
$N\mathbbm{1}=0$, since $T_{e_{i}}\mathbbm{1}=0$, $\forall i$.
Similarly, 
\begin{align}
N\Phi\left(k\right) & =\Phi\left(k\right)\label{eq:N1}\\
N\Phi\left(k\right)^{2} & =-2\left\Vert k\right\Vert ^{2}\mathbbm{1}+2\Phi\left(k\right)^{2},\quad\forall k\in\mathscr{L}.\label{eq:N2}
\end{align}
To see this, note that 
\begin{align*}
\sum_{i}T_{e_{i}}^{*}T_{e_{i}}\Phi\left(k\right) & =\sum_{i}T_{e_{i}}^{*}\left\langle e_{i},k\right\rangle \mathbbm{1}\\
 & =\sum_{i}\Phi\left(e_{i}\right)\left\langle e_{i},k\right\rangle \\
 & =\Phi\left(\sum_{i}\left\langle e_{i},k\right\rangle e_{i}\right)=\Phi\left(k\right),
\end{align*}
which is (\ref{eq:N1}). The verification of (\ref{eq:N2}) is similar. \end{example}
\begin{thm}
Let $\left\{ e_{i}\right\} $ be an ONB in $\mathscr{L}$, then 
\begin{equation}
T^{*}\overline{T}=\sum_{i}T_{e_{i}}^{*}T_{e_{i}}=N.\label{eq:MN1}
\end{equation}
\end{thm}
\begin{proof}
Note the span of $\left\{ e^{\Phi\left(k\right)}\mid k\in\mathscr{L}\right\} $
is dense in $L^{2}\left(\Omega,\mathbb{P}\right)$, and both sides
of (\ref{eq:MN1}) agree on $e^{\Phi\left(k\right)}$, $k\in\mathscr{L}$.
Indeed, by (\ref{eq:MN}), 
\[
T^{*}\overline{T}e^{\Phi\left(k\right)}=Ne^{\Phi\left(k\right)}=\left(\Phi\left(k\right)-\left\Vert k\right\Vert ^{2}\right)e^{\Phi\left(k\right)}.
\]
\end{proof}
\begin{cor}
Let $D:=T^{*}\overline{T}$. Specialize to the case of $n=1$, and
consider $F=f\left(\Phi\left(k\right)\right)$, $k\in\mathscr{L}$,
$f\in C^{\infty}\left(\mathbb{R}\right)$; then 
\begin{equation}
D\left(F\right)=-\left\Vert k\right\Vert _{\mathscr{L}}^{2}f''\left(\Phi\left(k\right)\right)+\Phi\left(k\right)f'\left(\Phi\left(k\right)\right).\label{eq:MTT}
\end{equation}
\end{cor}
\begin{proof}
A direct application of the formulas of $\overline{T}$ and $T^{*}$. \end{proof}
\begin{rem}
If $\left\Vert k\right\Vert _{\mathscr{L}}=1$ in (\ref{eq:MTT}),
then the RHS in (\ref{eq:MTT}) is obtained by a substitution of the
real valued random variable $\Phi\left(k\right)$ into the deterministic
function 
\begin{equation}
\delta\left(f\right):=-\left(\frac{d}{dx}\right)^{2}f+x\left(\frac{d}{dx}\right)f.
\end{equation}
Then eq. (\ref{eq:MTT}) may be rewritten as 
\begin{equation}
D\left(f\left(\Phi\left(k\right)\right)\right)=\delta\left(f\right)\circ\Phi\left(k\right),\quad f\in C^{\infty}\left(\mathbb{R}\right).\label{eq:H0}
\end{equation}
\end{rem}
\begin{cor}
If $\left\{ H_{n}\right\} _{n\in\mathbb{N}_{0}}$, $\mathbb{N}_{0}=\left\{ 0,1,2,\cdots\right\} $,
denotes the Hermite polynomials on $\mathbb{R}$, then we get for
$\forall k\in\mathscr{L}$, $\left\Vert k\right\Vert _{\mathscr{L}}=1$,
the following eigenvalues 
\begin{equation}
D\left(H_{n}\left(\Phi\left(k\right)\right)\right)=n\,H_{n}\left(\Phi\left(k\right)\right).\label{eq:H1}
\end{equation}
\end{cor}
\begin{proof}
It is well-known that the Hermite polynomials $H_{n}$ satisfies 
\begin{equation}
\delta\left(H_{n}\right)=n\,H_{n},\quad\forall n\in\mathbb{N}_{0},\label{eq:H2}
\end{equation}
and so (\ref{eq:H1}) follows from a substitution of (\ref{eq:H2})
into (\ref{eq:H0}).\end{proof}
\begin{thm}
The spectrum of $T^{*}\overline{T}$, as an operator in $L^{2}\left(\Omega,\mathcal{F},\mathbb{P}\right)$,
is as follows:
\[
spec_{L^{2}\left(\mathbb{P}\right)}\left(T^{*}\overline{T}\right)=\mathbb{N}_{0}=\left\{ 0,1,2,\cdots\right\} .
\]
\end{thm}
\begin{proof}
We saw that the $L^{2}\left(\mathbb{P}\right)$-representation is
unitarily equivalent to the Fock vacuum representation, and $\pi\left(\mbox{Fock-number operator}\right)=T^{*}\overline{T}$. 
\end{proof}

\subsection{Infinite-dimensional $\Delta$ and $\nabla_{\Phi}$}
\begin{cor}
Let $\left(\mathscr{L},\Omega,\mathcal{F},\mathbb{P},\Phi\right)$
be a Gaussian field, and let $T$ be the Malliavin derivative, $L^{2}\left(\Omega,\mathbb{P}\right)\xrightarrow{\;T\;}L^{2}\left(\Omega,\mathbb{P}\right)\otimes\mathscr{L}$.
Then, for all $F=p\left(\Phi\left(h_{1}\right),\cdots,\Phi\left(h_{n}\right)\right)\in\mathscr{D}$
(see \defref{MD}), we have 
\[
T^{*}T\left(F\right)=-\underset{\Delta F}{\underbrace{\sum_{i=1}^{n}\frac{\partial^{2}p}{\partial x_{i}}\left(\Phi\left(h_{1}\right),\cdots,\Phi\left(h_{n}\right)\right)}}+\underset{\nabla_{\Phi}F}{\underbrace{\sum_{i=1}^{n}\Phi\left(h_{i}\right)\frac{\partial p}{\partial x_{i}}\left(\Phi\left(h_{1}\right),\cdots,\Phi\left(h_{n}\right)\right)}},
\]
which is abbreviated 
\begin{equation}
T^{*}T=-\Delta+\nabla_{\Phi}.\label{eq:Dd1}
\end{equation}

\end{cor}
\begin{flushleft}
(For the general theory of infinite-dimensional Laplacians, see e.g.,
\cite{MR1966892}.)
\par\end{flushleft}
\begin{proof}
(Sketch) We may assume the system $\left\{ h_{i}\right\} _{i=1}^{n}\subset\mathscr{L}$
is orthonormal, i.e., $\left\langle h_{i},h_{j}\right\rangle =\delta_{ij}$.
Hence, for $F$ $F=p\left(\Phi\left(h_{1}\right),\cdots,\Phi\left(h_{n}\right)\right)\in\mathscr{D}$,
we have 
\[
TF=\sum_{i=1}^{n}\frac{\partial p}{\partial x_{i}}\left(\Phi\left(h_{1}\right),\cdots,\Phi\left(h_{n}\right)\right)\otimes h_{i},\;\mbox{and}
\]
\begin{eqnarray*}
T^{*}T\left(F\right) & = & -\sum_{i=1}^{n}\frac{\partial^{2}p}{\partial x_{i}^{2}}\left(\Phi\left(h_{1}\right),\cdots,\Phi\left(h_{n}\right)\right)\\
 &  & +\sum_{i=1}^{n}\Phi\left(h_{i}\right)\frac{\partial p}{\partial x_{i}}\left(\Phi\left(h_{1}\right),\cdots,\Phi\left(h_{n}\right)\right)
\end{eqnarray*}
which is the assertion. For details, see the proof of \thmref{MD}.\end{proof}
\begin{defn}
\label{def:Mgrad}Let $\left(\mathscr{L},\Omega,\mathcal{F},\mathbb{P},\Phi\right)$
be a Gaussian field. On the dense domain $\mathscr{D}\subset L^{2}\left(\Omega,\mathbb{P}\right)$,
we define the $\Phi$-gradient by
\begin{equation}
\nabla_{\Phi}F=\sum_{i=1}^{n}\Phi\left(h_{i}\right)\frac{\partial p}{\partial x_{i}}\left(\Phi\left(h_{1}\right),\cdots,\Phi\left(h_{n}\right)\right),\label{eq:Mgrad}
\end{equation}
for all $F=p\left(\Phi\left(h_{1}\right),\cdots,\Phi\left(h_{n}\right)\right)\in\mathscr{D}$.
(Note that $\nabla_{\Phi}$ is an unbounded operator in $L^{2}\left(\Omega,\mathbb{P}\right)$,
and $dom\left(\nabla_{\Phi}\right)=\mathscr{D}$.)\end{defn}
\begin{lem}
Let $\nabla_{\Phi}$ be the $\Phi$-gradient from \defref{Mgrad}.
The adjoint operator $\nabla_{\Phi}^{*}$, i.e., the $\Phi$-divergence,
is given as follows:
\begin{equation}
\nabla_{\Phi}^{*}\left(G\right)=\left(\sum_{i=1}^{n}\Phi\left(h_{i}\right)^{2}-n\right)G-\nabla_{\Phi}\left(G\right),\quad\forall G\in\mathscr{D}.\label{eq:Mdiv}
\end{equation}
\end{lem}
\begin{proof}
Fix $F,G\in\mathscr{D}$ as in \defref{MD}. Then $\exists n\in\mathbb{N}$,
$p,q\in\mathbb{R}\left[x_{1},\cdots,x_{n}\right]$, and $\left\{ h_{i}\right\} _{i=1}^{n}\subset\mathscr{L}$,
such that 
\begin{align*}
F & =p\left(\Phi\left(h_{1}\right),\cdots,\Phi\left(h_{n}\right)\right)\\
G & =q\left(\Phi\left(h_{1}\right),\cdots,\Phi\left(h_{n}\right)\right).
\end{align*}
Further assume that $\left\langle h_{i},h_{j}\right\rangle =\delta_{ij}$. 

In the calculation below, we use the following notation: $x=\left(x_{1},\cdots,x_{n}\right)\in\mathbb{R}^{n}$,
$dx=dx_{1}\cdots dx_{n}$ = Lebesgue measure, and $g_{n}=g^{G_{n}}$
= standard Gaussian distribution in $\mathbb{R}^{n}$, see (\ref{eq:p6}). 

Then, we have
\begin{eqnarray*}
 &  & \mathbb{E}\left(\left(\nabla_{\Phi}F\right)G\right)\\
 & = & \sum_{i=1}^{n}\mathbb{E}\left(\Phi\left(h_{i}\right)\frac{\partial p}{\partial x_{i}}\left(\Phi\left(h_{1}\right),\cdots,\Phi\left(h_{n}\right)\right)q\left(\Phi\left(h_{1}\right),\cdots,\Phi\left(h_{n}\right)\right)\right)\\
 & = & \sum_{i=1}^{n}\int_{\mathbb{R}^{n}}x_{i}\frac{\partial p}{\partial x_{i}}\left(x\right)q\left(x\right)g_{n}\left(x\right)dx\\
 & = & -\sum_{i=1}^{n}\int_{\mathbb{R}^{n}}p\left(x\right)\frac{\partial}{\partial x_{i}}\left(x_{i}q\left(x\right)g_{n}\left(x\right)\right)dx\\
 & = & -\sum_{i=1}^{n}\int_{\mathbb{R}^{n}}p\left(x\right)\left(q\left(x\right)+x_{i}\frac{\partial q}{\partial x_{i}}\left(x\right)-q\left(x\right)x_{i}^{2}\right)g_{n}\left(x\right)dx\quad\left(\frac{\partial g_{n}}{\partial x_{i}}=-x_{i}g_{n}\right)\\
 & = & \sum_{i=1}^{n}\mathbb{E}\left(FG\Phi\left(h_{i}\right)^{2}\right)-n\mathbb{E}\left(FG\right)-\mathbb{E}\left(F\nabla_{\Phi}G\right)\\
 & = & \mathbb{E}\left(FG\left(\sum\nolimits _{i=1}^{n}\Phi\left(h_{i}\right)^{2}-n\right)\right)-\mathbb{E}\left(F\nabla_{\Phi}G\right),
\end{eqnarray*}
which is the desired conclusion in (\ref{eq:Mdiv}).\end{proof}
\begin{rem}
Note $T_{k}^{*}$ is \emph{not} a derivation. In fact, we have 
\[
T_{k}^{*}\left(FG\right)=T_{k}^{*}\left(F\right)G+F\,T_{k}^{*}\left(G\right)-\Phi\left(k\right)FG,
\]
for all $F,G\in\mathscr{D}$, and all $k\in\mathscr{L}$. 

However, the divergence operator $\nabla_{\Phi}$ does satisfy the
Leibniz rule, i.e., 
\[
\nabla_{\Phi}\left(FG\right)=\left(\nabla_{\Phi}F\right)G+F\left(\nabla_{\Phi}G\right),\quad\forall F,G\in\mathscr{D}.
\]

\end{rem}

\subsection{Realization of the operators}
\begin{thm}
\label{thm:re}Let $\omega_{Fock}$ be the Fock state on $CCR\left(\mathscr{L}\right)$,
see (\ref{eq:cr4})-(\ref{eq:cr5}), and let $\pi_{F}$ denote the
corresponding (Fock space) representation, acting on $\Gamma_{sym}\left(\mathscr{L}\right)$,
see \lemref{dense}. Let $W:\Gamma_{sym}\left(\mathscr{L}\right)\longrightarrow L^{2}\left(\Omega,\mathbb{P}\right)$
be the isomorphism given by 
\begin{equation}
W\left(e^{k}\right):=e^{\Phi\left(k\right)-\frac{1}{2}\left\Vert k\right\Vert _{\mathscr{L}}^{2}},\quad k\in\mathscr{L}.\label{eq:re1}
\end{equation}
Here $L^{2}\left(\Omega,\mathbb{P}\right)$ denotes the Gaussian Hilbert
space corresponding to $\mathscr{L}$; see \defref{GH}. For vectors
$k\in\mathscr{L}$, let $T_{k}$ denote the Malliavin derivative in
the direction $k$; see \defref{MD}.

We then have the following realizations:
\begin{equation}
T_{k}=W\pi_{F}\left(a\left(k\right)\right)W^{*},\text{ and}\label{eq:re2}
\end{equation}
\begin{equation}
M_{\Phi\left(k\right)}-T_{k}=W\pi_{F}\left(a^{*}\left(k\right)\right)W^{*};\label{eq:re3}
\end{equation}
valid for all $k\in\mathscr{L}$, where the two identities (\ref{eq:re2})-(\ref{eq:re3})
hold on the dense domain $\mathscr{D}$ from \lemref{dense}.\end{thm}
\begin{rem}
The two formulas (\ref{eq:re2})-(\ref{eq:re3}) take the following
form, see Figs \ref{fig:o1}-\ref{fig:o2}.
\end{rem}
In the proof of the theorem, we make use of the following:
\begin{lem}
Let $\mathscr{L}$, $\mbox{CCR}\left(\mathscr{L}\right)$, and $\omega_{F}$
(= the Fock vacuum state) be as above. Then, for all $n,m\in\mathbb{N}$,
and all $h_{1},\cdots,h_{n}$, $k_{1},\cdots,k_{m}\in\mathscr{L}$,
we have the following identity:
\begin{eqnarray}
 &  & \omega_{F}\left(a\left(h_{1}\right)\cdots,a\left(h_{n}\right)a^{*}\left(k_{m}\right)\cdots a\left(k_{1}\right)\right)\nonumber \\
 & = & \delta_{n,m}\sum_{s\in S_{n}}\left\langle h_{1},k_{s\left(1\right)}\right\rangle _{\mathscr{L}}\left\langle h_{2},k_{s\left(2\right)}\right\rangle _{\mathscr{L}}\cdots\left\langle h_{n},k_{s\left(n\right)}\right\rangle _{\mathscr{L}}\label{eq:re3a}
\end{eqnarray}
where the summation on the RHS in (\ref{eq:re3a}) is over the symmetric
group $S_{n}$ of all permutations of $\left\{ 1,2,\cdots,n\right\} $.
(In the case of the CARs, the analogous expression on the RHS will
instead be a determinant.)\end{lem}
\begin{proof}
We leave the proof of the lemma to the reader; it is also contained
in \cite{MR611508}.\end{proof}
\begin{rem}
In physics-lingo, we say that the vacuum-state $\omega_{F}$ is determined
by its \emph{two-point functions}
\begin{align*}
\omega_{F}\left(a\left(h\right)a^{*}\left(k\right)\right) & =\left\langle h,k\right\rangle _{\mathscr{L}},\;\mbox{and}\\
\omega_{F}\left(a^{*}\left(k\right)a\left(h\right)\right) & =0,\quad\forall h,k\in\mathscr{L}.
\end{align*}
\end{rem}
\begin{proof}[Proof of \thmref{re}]
We shall only give the details for formula (\ref{eq:re2}). The modifications
needed for (\ref{eq:re3}) will be left to the reader. 

Since $W$ in (\ref{eq:re1}) is an isomorphic isomorphism, i.e.,
a unitary operator from $\Gamma_{sym}\left(\mathscr{L}\right)$ onto
$L^{2}\left(\Omega,\mathbb{P}\right)$, we may show instead that
\begin{equation}
T_{k}W=W\pi_{F}\left(a\left(k\right)\right)\label{eq:re4}
\end{equation}
holds on the dense subspace of all finite symmetric tensor polynomials
in $\Gamma_{sym}\left(\mathscr{L}\right)$; or equivalently on the
dense subspace in $\Gamma_{sym}\left(\mathscr{L}\right)$ spanned
by
\begin{equation}
\Gamma\left(l\right):=e^{l}:=\sum_{n=0}^{\infty}\frac{l^{\otimes n}}{\sqrt{n!}}\in\Gamma_{sym}\left(\mathscr{L}\right),\;l\in\mathscr{L};\label{eq:re5}
\end{equation}
see also \lemref{dense}. We now compute (\ref{eq:re4}) on the vectors
$e^{l}$ in (\ref{eq:re5}):
\begin{alignat*}{2}
T_{k}W\left(e^{l}\right) & =T_{k}\left(e^{\Phi\left(k\right)-\frac{1}{2}\left\Vert k\right\Vert _{\mathscr{L}}^{2}}\right) &  & \text{(by Lemma \ref{lem:dense})}\\
 & =e^{-\frac{1}{2}\left\Vert k\right\Vert _{\mathscr{L}}^{2}}T_{k}\big(e^{\Phi\left(k\right)}\big)\\
 & =e^{-\frac{1}{2}\left\Vert k\right\Vert _{\mathscr{L}}^{2}}\left\langle k,l\right\rangle _{\mathscr{L}}e^{\Phi\left(l\right)} & \quad & \text{(by Remark \ref{rem:MD})}\\
 & =W\pi_{F}\left(a\left(k\right)\right)\left(e^{l}\right),
\end{alignat*}
valid for all $k,l\in\mathscr{L}$. 
\end{proof}
\begin{figure}[H]
\[
\xymatrix{\Gamma_{sym}\left(\mathscr{L}\right)\ar[rr]^{W}\ar[d]_{\pi_{F}\left(a\left(k\right)\right)} &  & L^{2}\left(\Omega,\mathbb{P}\right)\ar[d]^{T_{k}}\\
\Gamma_{sym}\left(\mathscr{L}\right)\ar[rr]_{W} &  & L^{2}\left(\Omega,\mathbb{P}\right)
}
\]

\caption{\label{fig:o1}The first operator.}
\end{figure}

\begin{figure}[H]
\[
\xymatrix{\Gamma_{sym}\left(\mathscr{L}\right)\ar[rr]^{W}\ar[d]_{\pi_{F}\left(a^{*}\left(k\right)\right)} &  & L^{2}\left(\Omega,\mathbb{P}\right)\ar[d]^{M_{\Phi\left(k\right)}-T_{k}}\\
\Gamma_{sym}\left(\mathscr{L}\right)\ar[rr]_{W} &  & L^{2}\left(\Omega,\mathbb{P}\right)
}
\]

\caption{\label{fig:o2}The second operator.}
\end{figure}

\subsection{The unitary group}

For a given Gaussian field $\left(\mathscr{L},\Omega,\mathcal{F},\mathbb{P},\Phi\right)$,
we studied the $\mbox{CCR}\left(\mathscr{L}\right)$-algebra, and
the operators associated with its Fock-vacuum representation. 

From the determination of $\Phi$ by 
\begin{equation}
\mathbb{E}\big(e^{i\Phi\left(k\right)}\big)=e^{-\frac{1}{2}\left\Vert k\right\Vert _{\mathscr{L}}^{2}},\;k\in\mathscr{L};\label{eq:u1}
\end{equation}
we deduce that $\left(\Omega,\mathcal{F},\mathbb{P},\Phi\right)$
satisfies the following covariance with respect to the group $\mbox{Uni}\left(\mathscr{L}\right):=G\left(\mathscr{L}\right)$
of all unitary operators $U:\mathscr{L}\longrightarrow\mathscr{L}$. 

We shall need the following:
\begin{defn}
We say that $\alpha\in Aut\left(\Omega,\mathcal{F},\mathbb{P}\right)$
iff the following three conditions hold:
\begin{enumerate}
\item \label{enu:u1}$\alpha:\Omega\longrightarrow\Omega$ is defined $\mathbb{P}$
a.e. on $\Omega$, and $\mathbb{P}\left(\alpha\left(\Omega\right)\right)=1$.
\item $\mathcal{F}=\alpha\left(\mathcal{F}\right)$; more precisely, $\mathcal{F}=\left\{ \alpha^{-1}\left(B\right)\mid B\in\mathcal{F}\right\} $
where 
\begin{equation}
\alpha^{-1}\left(B\right)=\left\{ \omega\in\Omega\mid\alpha\left(\omega\right)\in B\right\} .\label{eq:u2}
\end{equation}

\item \label{enu:u3}$\mathbb{P}=\mathbb{P}\circ\alpha^{-1}$, i.e., $\alpha$
is a measure preserving automorphism.
\end{enumerate}

Note that when (\ref{enu:u1})-(\ref{enu:u3}) hold for $\alpha$,
then we have the unitary operators $U_{\alpha}$ in $L^{2}\left(\Omega,\mathcal{F},\mathbb{P}\right)$,
\begin{equation}
U_{\alpha}F=F\circ\alpha,\label{eq:u3}
\end{equation}
or more precisely, 
\[
\left(U_{\alpha}F\right)\left(\omega\right)=F\left(\alpha\left(\omega\right)\right),\;\mbox{a.e.}\;\omega\in\Omega,
\]
valid for all $F\in L^{2}\left(\Omega,\mathcal{F},\mathbb{P}\right)$.

\end{defn}
\begin{thm}
~
\begin{enumerate}
\item For every $U\in G\left(\mathscr{L}\right)$ (=\textup{ the unitary
group of $\mathscr{L}$}), there is a unique $\alpha\in Aut\left(\Omega,\mathcal{F},\mathbb{P}\right)$
s.t. 
\begin{equation}
\Phi\left(Uk\right)=\Phi\left(k\right)\circ\alpha,\label{eq:u4}
\end{equation}
or equivalently (see (\ref{eq:u3}))
\begin{equation}
\Phi\left(Uk\right)=U_{\alpha}\left(\Phi\left(k\right)\right),\;\forall k\in\mathscr{L}.\label{eq:u5}
\end{equation}

\item If $T:L^{2}\left(\Omega,\mathbb{P}\right)\longrightarrow L^{2}\left(\Omega,\mathbb{P}\right)\otimes\mathscr{L}$
is the Malliavin derivative from \defref{MD}, then we have:
\begin{equation}
TU_{\alpha}=\left(U_{\alpha}\otimes U\right)T.\label{eq:u6}
\end{equation}

\end{enumerate}
\end{thm}
\begin{proof}
The first conclusion in the theorem is immediate from the above discussion,
and we now turn to the covariance formula (\ref{eq:u6}). 

Note that (\ref{eq:u6}) involves unbounded operators, and it holds
on the dense subspace $\mathscr{D}$ in $L^{2}\left(\Omega,\mathbb{P}\right)$
from \lemref{dense}. Hence it is enough to verify (\ref{eq:u6})
on vectors in $L^{2}\left(\Omega,\mathbb{P}\right)$ of the form $e^{\Phi\left(k\right)-\frac{1}{2}\left\Vert k\right\Vert _{\mathscr{L}}^{2}}$,
$k\in\mathscr{L}$. Using \lemref{dense}, we then get:
\begin{alignat*}{2}
\mbox{LHS}_{\left(\ref{eq:u6}\right)}\big(e^{\Phi\left(k\right)-\frac{1}{2}\left\Vert k\right\Vert _{\mathscr{L}}^{2}}\big) & =e^{-\frac{1}{2}\left\Vert k\right\Vert _{\mathscr{L}}^{2}}T\big(e^{\Phi\left(Uk\right)}\big) &  & \text{(by \ensuremath{\left(\ref{eq:u4}\right)})}\\
 & =e^{-\frac{1}{2}\left\Vert Uk\right\Vert _{\mathscr{L}}^{2}}e^{\Phi\left(Uk\right)}\otimes\left(Uk\right) &  & \text{(by Remark \ref{rem:MD})}\\
 & =\left(U_{\alpha}\otimes U\right)\big(e^{\Phi\left(k\right)-\frac{1}{2}\left\Vert k\right\Vert _{\mathscr{L}}^{2}}\big) & \quad\\
 & =\mbox{RHS}_{\left(\ref{eq:u6}\right)}
\end{alignat*}

\end{proof}

\section{\label{sec:CCR}The Fock-state, and representation of CCR, realized
as Malliavin calculus}

We now resume our analysis of the representation of the canonical
commutation relations (CCR)-algebra induced by the canonical Fock
state (see (\ref{eq:cr2})). In our analysis below, we shall make
use of the following details: Brownian motion, It\=o-integrals, and
the Malliavin derivative.

\textbf{The general setting.} Let $\mathscr{L}$ be a fixed Hilbert
space, and let $\mbox{CCR}\left(\mathscr{L}\right)$ be the $*$-algebra
on the generators $a\left(k\right)$, $a^{*}\left(l\right)$, $k,l\in\mathscr{L}$,
and subject to the relations for the CCR-algebra, see \subref{ccr}:
\begin{align}
\left[a\left(k\right),a\left(l\right)\right] & =0,\quad\mbox{and}\label{eq:ac1}\\
\left[a\left(k\right),a^{*}\left(l\right)\right] & =\left\langle k,l\right\rangle _{\mathscr{L}}\mathbbm{1}\label{eq:ac2}
\end{align}
where $\left[\cdot,\cdot\right]$ is the commutator bracket.

A representation $\pi$ of $\mbox{CCR}\left(\mathscr{L}\right)$ consists
of a fixed Hilbert space $\mathscr{H}=\mathscr{H}_{\pi}$ (the representation
space), a dense subspace $\mathscr{D}_{\pi}\subset\mathscr{H}_{\pi}$,
and a $*$-homomorphism $\pi:\mbox{CCR}\left(\mathscr{L}\right)\longrightarrow End\left(\mathscr{D}_{\pi}\right)$
such that 
\begin{equation}
\mathscr{D}_{\pi}\subset dom\left(\pi\left(A\right)\right),\quad\forall A\in\mbox{CCR}.\label{eq:ac3}
\end{equation}
The representation axiom entails the commutator properties resulting
from (\ref{eq:ac1})-(\ref{eq:ac2}); in particular $\pi$ satisfies
\begin{align}
\left[\pi\left(a\left(k\right)\right),\pi\left(a\left(l\right)\right)\right]F & =0,\quad\mbox{and}\label{eq:ac4}\\
\left[\pi\left(a\left(k\right)\right),\pi\left(a\left(l\right)\right)^{*}\right]F & =\left\langle k,l\right\rangle _{\mathscr{L}}F,\label{eq:ac5}
\end{align}
$\forall k,l\in\mathscr{L}$, $\forall F\in\mathscr{D}_{\pi}$; where
$\pi\left(a^{*}\left(l\right)\right)=\pi\left(a\left(l\right)\right)^{*}$
. 

In the application below, we take $\mathscr{L}=L^{2}\left(0,\infty\right)$,
and $\mathscr{H}_{\pi}=L^{2}\left(\Omega,\mathcal{F}_{\Omega},\mathbb{P}\right)$
where $\left(\Omega,\mathcal{F}_{\Omega},\mathbb{P}\right)$ is the
standard Wiener probability space, and 
\begin{equation}
\Phi_{t}\left(\omega\right)=\omega\left(t\right),\quad\forall\omega\in\Omega,\:t\in[0,\infty).\label{eq:ac6}
\end{equation}
For $k\in\mathscr{L}$, we set 
\[
\Phi\left(k\right)=\int_{0}^{\infty}k\left(t\right)d\Phi_{t}\;(\mbox{=the It\={o}-integral}.)
\]

The dense subspace $\mathscr{D}_{\pi}\subset\mathscr{H}_{\pi}$ is
generated by the polynomial fields: 

For $n\in\mathbb{N}$, $h_{1},\cdots,h_{n}\in\mathscr{L}=L_{\mathbb{R}}^{2}\left(0,\infty\right)$,
$p\in\mathbb{R}^{n}\longrightarrow\mathbb{R}$ a polynomial in $n$
real variables, set 
\begin{equation}
F=p\left(\Phi\left(h_{1}\right),\cdots,\Phi\left(h_{n}\right)\right),\quad\mbox{and}\label{eq:ac7}
\end{equation}
\begin{equation}
\pi\left(a\left(k\right)\right)F=\sum_{j=1}^{n}\left(\frac{\partial}{\partial x_{j}}p\right)\left(\Phi\left(h_{1}\right),\cdots,\Phi\left(h_{n}\right)\right)\left\langle h_{j},k\right\rangle .\label{eq:ac8}
\end{equation}

It follows from \lemref{md} that $\mathscr{D}_{\pi}$ is an algebra
under pointwise product and that
\begin{equation}
\pi\left(a\left(k\right)\right)\left(FG\right)=\left(\pi\left(a\left(k\right)\right)F\right)G+F\left(\pi\left(a\left(k\right)\right)G\right),\label{eq:ac9}
\end{equation}
$\forall k\in\mathscr{L}$, $\forall F,G\in\mathscr{D}_{\pi}$. Equivalently,
$T_{k}:=\pi\left(a\left(k\right)\right)$ is a derivation in the algebra
$\mathscr{D}_{\pi}$ (relative to pointwise product.)
\begin{thm}
\label{thm:CM}With the operators $\pi\left(a\left(k\right)\right)$,
$k\in\mathscr{L}$, we get a $*$-representation $\pi:\mbox{CCR}\left(\mathscr{L}\right)\longrightarrow End\left(\mathscr{D}_{\pi}\right)$,
i.e., $\pi\left(a\left(k\right)\right)$ = the Malliavin derivative
in the direction $k$, 
\begin{equation}
\pi\left(a\left(k\right)\right)F=\left\langle T\left(F\right),k\right\rangle _{\mathscr{L}},\quad\forall F\in\mathscr{D}_{\pi},\:\forall k\in\mathscr{L}.\label{eq:ac10}
\end{equation}
\end{thm}
\begin{proof}
We begin with the following\end{proof}
\begin{lem}
\label{lem:aa}Let $\pi$, $\mbox{CCR}\left(\mathscr{L}\right)$,
and $\mathscr{H}_{\pi}=L^{2}\left(\Omega,\mathcal{F}_{\Omega},\mathbb{P}\right)$
be as above. For $k\in\mathscr{L}$, we shall identify $\Phi\left(k\right)$
with the unbounded multiplication operator in $\mathscr{H}_{\pi}$:
\begin{equation}
\mathscr{D}_{\pi}\ni F\longmapsto\Phi\left(k\right)F\in\mathscr{H}_{\pi}.\label{eq:ac11}
\end{equation}
For $F\in\mathscr{D}_{\pi}$, we have $\pi\left(a\left(k\right)\right)^{*}F=-\pi\left(a\left(k\right)\right)F+\Phi\left(k\right)F$;
or in abbreviated form: 
\begin{equation}
\pi\left(a\left(k\right)\right)^{*}=-\pi\left(a\left(k\right)\right)+\Phi\left(k\right)\label{eq:ac12}
\end{equation}
valid on the dense domain $\mathscr{D}_{\pi}\subset\mathscr{H}_{\pi}$. \end{lem}
\begin{proof}
This follows from the following computation for $F,G\in\mathscr{D}_{\pi}$,
$k\in\mathscr{L}$. 

Setting $T_{k}:=\pi\left(a\left(k\right)\right)$, we have 
\[
\mathbb{E}\left(T_{k}\left(F\right)G\right)+\mathbb{E}\left(F\,T_{k}\left(G\right)\right)=\mathbb{E}\left(T_{k}\left(FG\right)\right)=\mathbb{E}\left(\Phi\left(k\right)FG\right).
\]
Hence $\mathscr{D}_{\pi}\subset dom\left(T_{k}^{*}\right)$, and $T_{k}^{*}\left(F\right)=-T_{k}\left(F\right)+\Phi\left(k\right)F$,
which is the desired conclusion (\ref{eq:ac12}). 
\end{proof}

\begin{proof}[Proof of \thmref{CM} continued]
 It is clear that the operators $T_{k}=\pi\left(a\left(k\right)\right)$
form a commuting family. Hence on $\mathscr{D}_{\pi}$, we have for
$k,l\in\mathscr{L}$, $F\in\mathscr{D}_{\pi}$:
\begin{alignat*}{2}
\left[T_{k},T_{l}^{*}\right]\left(F\right) & =\left[T_{k},\Phi\left(l\right)\right]\left(F\right) &  & \text{by }\left(\ref{eq:ac12}\right)\\
 & =T_{k}\left(\Phi\left(l\right)F\right)-\Phi\left(l\right)\left(T_{k}\left(F\right)\right) & \qquad\\
 & =T_{k}\left(\Phi\left(l\right)\right)F &  & \text{by \ensuremath{\left(\ref{eq:ac9}\right)}}\\
 & =\left\langle k,l\right\rangle _{\mathscr{L}}F &  & \text{by \ensuremath{\left(\ref{eq:ac8}\right)}}
\end{alignat*}
which is the desired commutation relation (\ref{eq:ac2}). 

The remaining check on the statements in the theorem are now immediate. \end{proof}
\begin{cor}
\label{cor:ccrs}The state on $\mbox{CCR}\left(\mathscr{L}\right)$
which is induced by $\pi$ and the constant function $\mathbbm{1}$
in $L^{2}\left(\Omega,\mathbb{P}\right)$ is the Fock-vacuum-state,
$\omega_{Fock}$. \end{cor}
\begin{proof}
The assertion will follow once we verify the following two conditions:
\begin{equation}
\int_{\Omega}T_{k}^{*}T_{k}\left(\mathbbm{1}\right)d\mathbb{P}=0\label{eq:acc1}
\end{equation}
and 
\begin{equation}
\int_{\Omega}T_{k}T_{l}^{*}\left(\mathbbm{1}\right)d\mathbb{P}=\left\langle k,l\right\rangle _{\mathscr{L}}\label{eq:acc2}
\end{equation}
for all $k,l\in\mathscr{L}$. 

This in turn is a consequence of our discussion of eqs (\ref{eq:cr4})-(\ref{eq:cr5})
above: The Fock state $\omega_{Fock}$ is determined by these two
conditions. The assertions (\ref{eq:acc1})-(\ref{eq:acc2}) follow
from $T_{k}\left(\mathbbm{1}\right)=0$, and $\left(T_{k}T_{l}^{*}\right)\left(\mathbbm{1}\right)=\left\langle k,l\right\rangle _{\mathscr{L}}\mathbbm{1}$.
See (\ref{eq:mc6}).
\end{proof}

\begin{cor}
For $k\in L_{\mathbb{R}}^{2}\left(0,\infty\right)$ we get a family
of selfadjoint multiplication operators $T_{k}+T_{k}^{*}=M_{\Phi\left(k\right)}$
on $\mathscr{D}_{\pi}$ where $T_{k}=\pi\left(a\left(k\right)\right)$.
Moreover, the von Neumann algebra generated by these operators is
$L^{\infty}\left(\Omega,\mathbb{P}\right)$, i.e., the maximal abelian
$L^{\infty}$-algebra of all multiplication operators in $\mathscr{H}_{\pi}=L^{2}\left(\Omega,\mathbb{P}\right)$.\end{cor}
\begin{rem}
In our considerations of representations $\pi$ of $\text{CCR}\left(\mathscr{L}\right)$
in a Hilbert space $\mathscr{H}_{\pi}$, we require the following
five axioms satisfied:
\begin{enumerate}
\item \label{enu:c1}a dense subspace $\mathscr{D}_{\pi}\subset\mathscr{H}_{\pi}$; 
\item $\pi:\text{CCR}\left(\mathscr{L}\right)\longrightarrow End\left(\mathscr{D}_{\pi}\right)$,
i.e., $\mathscr{D}_{\pi}\subset\cap_{A\in\text{CCR}\left(\mathscr{L}\right)}dom\left(\pi\left(A\right)\right)$; 
\item \label{enu:c3}$\left[\pi\left(a\left(k\right)\right),\pi\left(a\left(l\right)\right)\right]=0$,
$\forall k,l\in\mathscr{L}$; 
\item $\left[\pi\left(a\left(k\right)\right),\pi\left(a\left(l\right)\right)^{*}\right]=\left\langle k,l\right\rangle _{\mathscr{L}}I_{\mathscr{H}_{\pi}}$,
$\forall k,l\in\mathscr{L}$; and 
\item \label{enu:c5}$\pi\left(a^{*}\left(k\right)\right)\subset\pi\left(a\left(k\right)\right)^{*}$,
$\forall k\in\mathscr{L}$. 
\end{enumerate}

Note that in our assignment for the operators $\pi\left(a\left(k\right)\right)$,
and $\pi\left(a^{*}\left(k\right)\right)$ in \lemref{aa}, we have
all the conditions (\ref{enu:c1})-(\ref{enu:c5}) satisfied. We say
that $\pi$ is a \emph{selfadjoint representation}.

If alternatively, we define 
\begin{equation}
\rho:\text{CCR}\left(\mathscr{L}\right)\longrightarrow End\left(\mathscr{D}_{\pi}\right)\label{eq:rh1}
\end{equation}
with the following modification:
\begin{equation}
\left\{ \begin{split}\rho\left(a\left(k\right)\right) & =T_{k},\;k\in\mathscr{L},\;\mbox{and}\\
\rho\left(a^{*}\left(k\right)\right) & =\Phi\left(k\right)
\end{split}
\right.\label{eq:rh2}
\end{equation}
then this $\rho$ will satisfy (\ref{enu:c1})-(\ref{enu:c3}), and
\[
\left[\rho\left(a\left(k\right)\right),\rho\left(a^{*}\left(l\right)\right)\right]=\left\langle k,l\right\rangle _{\mathscr{L}}I_{\mathscr{H}_{\pi}};
\]
but then $\rho\left(a\left(k\right)\right)\varsubsetneqq\rho\left(a\left(k\right)\right)^{*}$;
i.e., non-containment of the respective graphs. 

One generally says that the representation $\pi$ is (formally) selfadjoint,
while the second representation $\rho$ is \emph{not.}
\end{rem}

\section{\label{sec:con}Conclusions: the general case}
\begin{defn}
A representation $\pi$ of $\mbox{CCR}\left(\mathscr{L}\right)$ is
said to be \emph{admissible} iff (Def.) $\exists\left(\Omega,\mathcal{F},\mathbb{P}\right)$
as above such that $\mathscr{H}_{\pi}=L^{2}\left(\Omega,\mathcal{F},\mathbb{P}\right)$,
and there exists a linear mapping $\Phi:\mathscr{L}\longrightarrow L^{2}\left(\Omega,\mathcal{F},\mathbb{P}\right)$
subject to the condition:

For every $n\in\mathbb{N}$, and every $k,h_{1},\cdots,h_{n}\in\mathscr{L}$,
the following holds on its natural dense domain in $\mathscr{H}_{\pi}$:
For every $p\in\mathbb{R}\left[x_{1},\cdots,x_{n}\right]$, we have
\begin{equation}
\pi\left(\left[a\left(k\right),p\left(a^{*}\left(h_{1}\right),\cdots,a^{*}\left(h_{n}\right)\right)\right]\right)=\sum_{i=1}^{n}\left\langle k,h_{i}\right\rangle _{\mathscr{L}}M_{\frac{\partial p}{\partial x_{i}}\left(\Phi\left(h_{1}\right),\cdots,\Phi\left(h_{n}\right)\right)},\label{eq:ar1}
\end{equation}
with the $M$ on the RHS denoting ``multiplication.''\end{defn}
\begin{cor}
~
\begin{enumerate}
\item \label{enu:ar1}Every admissible representation $\pi$ of $\mbox{CCR}\left(\mathscr{L}\right)$
yields an associated Malliavin derivative as in (\ref{eq:ar1}).
\item \label{enu:ar2}The Fock-vacuum representation $\pi_{F}$ is admissible. 
\end{enumerate}
\end{cor}
\begin{proof}
(\ref{enu:ar1}) follows from the definition combined with \corref{li}.
(\ref{enu:ar2}) is a direct consequence of \lemref{ST} and \thmref{MD};
see also \corref{ccrs}. \end{proof}
\begin{acknowledgement*}
The co-authors thank the following colleagues for helpful and enlightening
discussions: Professors Sergii Bezuglyi, Ilwoo Cho, Paul Muhly, Myung-Sin
Song, Wayne Polyzou, and members in the Math Physics seminar at The
University of Iowa.

\bibliographystyle{amsalpha}
\bibliography{ref}
\end{acknowledgement*}

\end{document}